\documentclass[11pt]{amsart}
\usepackage{bm}
\usepackage{mathrsfs}
\allowdisplaybreaks
\usepackage{hyperref}

\makeatletter

\renewcommand{\tocsection}[3]{%
	\indentlabel{\@ifnotempty{#2}{\bfseries\ignorespaces#1 #2\quad}}\bfseries#3}

\usepackage{amssymb}
\usepackage{multicol}
\usepackage{stmaryrd}
\usepackage{cite}
\usepackage{epsfig}
\usepackage{color}
\newcommand\bes{\begin{eqnarray}}
	\newcommand\ees{\end{eqnarray}}

\newtheorem{theorem}{Theorem}[section]
\newtheorem{lemma}[theorem]{Lemma}

\newtheorem{proposition}[theorem]{Proposition}
\numberwithin{equation}{section}

\theoremstyle{plain}
\newtheorem*{proposition*}{Proposition A}

\newcommand\bess{\begin{eqnarray*}}
	\newcommand\eess{\end{eqnarray*}}

\theoremstyle{definition}  
\newtheorem{remark}[theorem]{Remark}

\begin{document}
	
	\title{The high dimensional Monostable reaction-diffusion
		equation  with free boundary and radial 
		symmetry }
	\author[Hongkai Cao]{Hongkai Cao} \author[Jingyi Cui]{Jingyi Cui} \author[Chengzhe Tang]{Chengzhe Tang} \author[Xiaoyan Zhang]{Xiaoyan Zhang}
	\thanks{\hspace{-.5cm}
		School of Mathematics, Shandong University, Jinan 250100, China.
		\\
		\mbox{\ \  Emails:} chk@mail.sdu.edu.cn (H. Cao),\ \ zxysd@sdu.edu.cn (X. Zhang),\ \ 202411888@mail.sdu.edu.cn(C. Tang),\ \ c15233069201@163.com(J. Cui)}
	\thanks{ 
		The research was partially supported by the National Natural Science Foundation of China (No.11571200) and the Natural Science Foundation of Shandong Province (No. ZR2021MA062).
	}
	
	\date{\today}

	\begin{abstract}
We consider the radially symmetric version of the reaction-diffusion equation $u_t-d\Delta u=f(u)$ with a monostable nonlinearity $f$, viewed as a model for the spreading of a species with population range $r<h(t)$ and density $u(t,r)$ ($r=|x|$), where the free boundary $r=h(t)$ is governed by $u(t,h(t))=\delta>0$ and $h'(t)=-d u_r(t,h(t))/\delta$. For the one-dimensional case ($N=1$), Du \cite{DN} proved that when $\delta\in(0,1)$, spreading occurs: $u\to1$ locally uniformly in $\mathbb{R}$, $h(t)\to\infty$, and $\lim_{t\to\infty}[h(t)-c_*t]=\tilde{h}\in\mathbb{R}$ with no logarithmic shift. In the present paper we consider $N\ge2$ and establish a complete trichotomy: spreading for $\delta\in(0,1)$; transition for $\delta=1$, where $u\to1$  uniformly  on  $[0,h(t)]$ and $h(t)\to h_\infty\in(0,\infty)$; and vanishing for $\delta>1$, where $h(t)\to0$ and $u\to\delta$ uniformly on $[0,h(t)]$.
 In the vanishing regime, we further obtain 	$h(t)\sim H_\infty e^{f(\delta)t/(N\delta)}$, where $H_\infty>0$.
 
 For the spreading regime, by constructing sharp upper and lower solutions, we prove that the solution converges globally to the semi-wave profile and reveal a logarithmic shift of the form 
\[\lim_{t\to\infty}\big[h(t)-c_*t+c_N(\delta)\log t\big]=\hat{h}\in\mathbb{R}, \ \mbox{ with }\
\lim_{\delta\to0}c_N(\delta)=d(N-1)/c_0,\]
where $c_0$ is the minimal travelling-wave speed of the corresponding Cauchy problem.

		\bigskip

		\medskip
		
		
	\end{abstract}
	
	\maketitle

	\section{Introduction and Basic results}

	\subsection{Background}

	In this paper, we mainly consider the following problem
	\begin{equation}\label{eq1.1}
		\left\{\begin{array}{ll}
			u_{t}-d (u_{r r}+\frac{N-1}{r}u_r)=f(u), & t>0, \ 0<r<h(t) ,\\
			u_r(t, 0)=0, \ u(t, h(t))=\delta, & t>0 ,\\
			h^{\prime}(t)=-\frac{d}{\delta} u_{r}(t, h(t)), & t>0 ,\\
		h(0)=h_{0}, u(0, r)=u_{0}(r), & 0 \leq r \leq h_{0},
		\end{array}\right.
	\end{equation}
	where $N\geq 2$, $d$, $h_{0}$, $\delta>0$,   and 
	the initial function $u_{0}(r)$ is assumed to belong to $\mathcal{X}(h_{0})$ given by
	\[\mathcal{X}\left(h_{0}\right):=\left\{\phi \in C^{2}\left(\left[0, h_{0}\right]\right): \phi(r)>0 \text { in }\left[0, h_{0}\right],\,\phi_r\left(0\right)= 0,\ \phi\left( h_{0}\right)=\delta\right\},\]
and	$f$ is a monostable function
\[
(\mathbf{f_m}): \left\{ f \in C^1 \;\middle|\; f(0)=f(1)=0,\; f'(0)>0>f'(1),\; (1-u)f(u)>0 \ \text{for } u>0,\ u\neq 1 \right\}
\]

Our aim is to study the long-time dynamical behavior of species invasion in a high-dimensional radially symmetric setting, where $u$ denotes the species density, $d$ is the diffusion coefficient,  $N\geq 2$ is the spatial dimension and $[0, h(t)]$ is the population range. The free boundary condition
 arises from the ``preferred population density'', at which the species maintains a preferred density $\delta>0$; see \cite{DN} for further details.

It is well known that the Cauchy problem 
\begin{equation}\label{eq1.2}
	\begin{cases}
		U_t - d\Delta U = f(U) & \text{for } x \in \mathbb{R}^N, \ t > 0, \\
		U(0, x) = U_0(x) & \text{for } x \in \mathbb{R}^N,
	\end{cases}
\end{equation}
 was studied in the seminal works \cite{Fisher1937} and \cite{KPP1937}, where \( U_0 \) is nonnegative and has nonempty compact support.
For a fixed level $\theta\in(0,1)$, define
\[
H_\theta(t):=\{x:U(t,x)=\theta\}.
\]
 By \cite{AW1978}, when spreading occurs, for any $\epsilon>0$ there exists $T>0$ such that
\begin{equation}
H_\theta(t)\subset\{x\in\mathbb R^N:(c_0-\epsilon)t\leq|x|\leq(c_0+\epsilon)t\},\qquad t\geq T.
\end{equation}
The constant \( c_0 \), usually called the spreading speed of \eqref{eq1.2}, is determined by the well-known traveling wave problem
\begin{equation}\label{eq1.4}
	dQ'' - cQ' + f(Q) = 0, \quad Q > 0 \text{ in } \mathbb{R}, \quad Q(-\infty) = 0, \quad Q(+\infty) = 1, \quad Q(0) = 1/2. 
\end{equation}
More precisely, a solution \( Q_c \) of \eqref{eq1.4} exists if and only if \( c \geq c_0 \).

Logarithmic shifts in reaction-diffusion propagation have a rich history. The pulled monostable case is characterized by $c_0=2\sqrt{df'(0)}$; see \cite{Gil}. For the pulled nonlinearities and radial initial data considered in \cite{JG,RRR}, with $N\geq2$, the level position $R_\theta(t):=\sup\{r\geq0:U(t,r)\geq\theta\}$ satisfies
\begin{equation}
R_\theta(t)=c_0t-\frac{d(N+2)}{c_0}\log t+O(1)
\end{equation}
for each fixed $\theta\in(0,1)$. For a monostable nonlinearity in the pushed case, characterized by $c_0>2\sqrt{df'(0)}$, see \cite{KU}, the corresponding radial level position instead satisfies
\begin{equation}\label{eq1.6}
R_\theta(t)=c_0t-\frac{d(N-1)}{c_0}\log t+O(1).
\end{equation}

\begin{proposition*}[Proposition 1.3 of \cite{DN}]
	Suppose that \( \mathbf{(f_m)} \) holds  and \( \delta \in (0, 1) \). Then there exists a unique pair  
	$
	(c, q) = (c_*, q_*)$ with $c_* = c_*^\delta > 0$ satisfying
	\begin{equation}\label{eq1.7}
		\begin{cases} 
		dq'' - cq' + f(q) = 0, \quad q > 0 \text{ in } (0, \infty), \\ 
		q(0) = \delta,\quad  q(\infty) = 1, \quad q'(0) = c \frac{\delta}{d}.
	\end{cases}
\end{equation}	
 Moreover, \( q_*'(z) > 0 \) for \( z \geq 0 \), \( c_* < c_0 \), and  
	$
	\lim_{\delta \to 0^+} c_*^\delta = c_0$.
\end{proposition*} 
In \cite{DN}, the author considers the one-dimensional case $N=1$ with $\delta\in(0,1)$ and proves that there exists $\tilde{h}$ such that
\[
\begin{cases}
	\displaystyle \lim_{t\to\infty} [h(t) - c_*t] = \tilde{h},\quad \displaystyle \lim_{t\to\infty} h'(t) = c_*, \\[1.2ex]
	\displaystyle \lim_{t\to\infty} \sup_{x\in[0,h(t)]} |u(t,x) - q_*(h(t)-x)| = 0.
\end{cases}
\]
 In \cite{DLNS}, the model was extended to nonlinearities of Allee type, including bistable and combustion nonlinearities. The time-periodic problem was studied in \cite{MDW}, where spreading and its asymptotic speed were established under a KPP assumption, and in \cite{DMW}, where precise convergence to a semi-wave was obtained for general monostable nonlinearities. Shifting environments and advection were considered in \cite{DHL} and \cite{CMZ}, respectively. For $\delta>1$, convergence to a receding wave on a semi-infinite interval was proved in \cite{CDN}. With a prescribed density below $\delta$ at the fixed boundary, a trichotomy of spreading, finite-time vanishing and transition was established in \cite{CDNZ}. However, the high-dimensional radially symmetric problem \eqref{eq1.1} remained open.

For comparison, consider the classical Stefan free boundary problem with boundary conditions $u(t,h(t))=0$ and $h'(t)=-\mu_0u_r(t,h(t))$, where $\mu_0>0$ is fixed. In the spreading regime, \cite{DL,DZ,DHZ} give constants $c^*>0$, $\zeta>0$, and $H\in\mathbb R$ such that
\[
\left\{
\begin{aligned}
&\lim_{t\to\infty}\left[h(t)-c^*t+\frac{d(N-1)}{\zeta c^*}\log t\right]=H,\\
&\lim_{t\to\infty}\sup_{r\in[0,h(t)]}
\left|u(t,r)-q^{c^*}\left(h(t)-r\right)\right|=0.
\end{aligned}
\right.
\]
Here $q^{c^*}$ is the corresponding semi-wave, and $c^*$ and $\zeta$ are independent of $N$.

In this paper, we consider problem \eqref{eq1.1} for $N\ge 2$ and establish a trichotomy result in terms of $\delta$: when $\delta\in(0,1)$, spreading occurs, i.e., $h(t)\to\infty$ and $u(t,r)\to 1$ locally uniformly on $[0,\infty)$ as $t\to\infty$; when $\delta=1$, transition occurs, i.e., $h(t)\to h_\infty\in(0,\infty)$ and $u(t,r)\to 1$ uniformly as $t\to\infty$; and when $\delta>1$, vanishing occurs, i.e., $h(t)\to 0$ as $t\to\infty$. We further determine its precise exponential decay rate, $-f(\delta)/(N\delta)$, which exhibits the dependence on the spatial dimension. Moreover, for $\delta\in(0,1)$, we construct refined upper and lower solutions via perturbation of the semi-wave problem \eqref{eq1.7}, which yields a logarithmic shift different from that in the one-dimensional case $N=1$.

	\subsection{Main result and arrangement}

	Our main results concerning \eqref{eq1.1} are listed below.
	\begin{theorem}[\underline{Existence and uniqueness of solution}]\label{th1.1}
		Suppose that $(\mathbf{f_m})$ holds and $\delta>0$. Then, for any given $u_0 \in \mathcal{X}(h_0)$ and $\alpha \in (0, 1)$, \eqref{eq1.1} admits a unique solution $(u,h)$ for $t\in(0, \infty)$
		\[
		(u,  h) \in C^{1+\frac{\alpha}{2}, 2+\alpha}(\Omega_{\infty}) \times C^{1+\frac{\alpha}{2}}((0, \infty)),
		\]
		where 
		\[
		\Omega_{\infty} := \{(t, r) \in \mathbb{R}^2 : t \in (0, \infty), \, r \in [0, h(t)]\} .
		\]
	\end{theorem}

	\begin{theorem}[\underline{Long-time dynamics}]
		Suppose that the condition of Theorem~\ref{th1.1} holds.
		\begin{enumerate}
			\item If $0<\delta<1$, 
		then spreading occurs:
					\[
				\lim_{t \to \infty}	h(t)=\infty,\quad \text{and}\quad \lim_{t \to \infty}u(t,r)= 1\ \text{locally uniformly in } [0, \infty).
					\]
				\item If $\delta=1$, then transition happens:
				\[\lim_{t \to \infty}h(t)= h_\infty\in(0,\infty)\quad \text{and}\quad \lim_{t\to\infty}u(t,r)=1\ \text{ uniformly in } [0, h(t)].\]
				\item If $\delta>1$, then  vanishing happens:
				\[\lim_{t\to\infty} h(t)=0\quad \text{and}\quad \lim_{t\to\infty}u(t,r)=\delta\ \text{ uniformly in } [0, h(t)]. \]
Moreover, there exists $H_\infty>0$ such that
\[
h(t)=H_\infty e^{-\kappa t}\bigl(1+O(e^{-\kappa t})\bigr),\qquad
\kappa=-\frac{f(\delta)}{N\delta}>0,
\quad t\to\infty.
\]

			\end{enumerate}
	\end{theorem}

	To more precisely describe the long-time behavior of the species when spreading happens, we require the following result. 
	\begin{proposition}[\underline{Semi-wave}]\label{pr1.3}
		Suppose that the condition of Theorem~\ref{th1.1} holds with $\delta\in(0,1)$.  There exists $\mu^*>\mu_0:=\frac{d}{\delta}$ such that 
		if $\mu\in(0, \mu^*]$, there exists a unique pair $(c, q) = (c_*(\mu), q_{c_*(\mu)})$ satisfying
		\begin{equation}\label{eq1.8}
			\begin{cases}
				d q'' - c q' + f(q) = 0, \quad q' > 0 \quad \text{in } (0, \infty), \\[4pt]
				q(0) = \delta, \quad q(\infty) = 1, \quad q'(0) = \frac{c}{\mu}.
			\end{cases}
		\end{equation}
		Moreover,
		$ c_*(\mu) \in (0, c_0)$ and $c_*(\mu_0)=c_*^\delta$.
	\end{proposition}

	With the help of Proposition \ref{pr1.3}, we can obtain a more precise long-time behavior below. In particular, $c_*(\mu_0)$ is the asymptotic spreading speed  $h(t)$.
	\begin{theorem}[\underline{Precise propagation profile}]\label{th1.4}
		Suppose that the condition of Theorem \ref{th1.1} holds and $\delta\in(0,1)$. Then there exists $\hat{h}\in \mathbb{R}$ such that 
		\[
		\begin{cases}
			\lim_{t \to \infty} [h(t) - c_* t+c_N (\delta) \log t] = \hat{h}, \quad \lim_{t \to \infty} h'(t) = c_*, \\
			\lim_{t \to \infty} \sup_{r \in [0, h(t)]} \left|u(t, r) - q_*\left( h(t) - r\right)\right| = 0, \\
		\end{cases}
		\]
		where $(c_*=c_*(\mu_0)=c_*^\delta,q_*=q_{c_*(\mu_0)})$ is the unique solution of \eqref{eq1.8} with $\mu=\mu_0=\frac{d}{\delta}$ and 
\[
c_N (\delta)= d(N-1) c_*^{-1}
\left[
1 + \delta^2 c_*(\mu_0)
\left(
d \int_0^\infty
\bigl( q_{c_*(\mu_0)}'(z) \bigr)^2
e^{-c_*(\mu_0) d^{-1} z} \, dz
\right)^{-1}
\right]^{-1}.
\]

	\end{theorem}

\begin{proposition}\label{pr1.5}
Assume $(\mathbf{f_m})$. The coefficient $c_N(\delta)$ in Theorem \ref{th1.4} satisfies
\[
\lim_{\delta\downarrow0}c_N(\delta)=\frac{d(N-1)}{c_0},
\]
where $c_0$ is the minimal travelling-wave speed of \eqref{eq1.2}.
\end{proposition}

\begin{remark}
	We observe that this logarithmic shift is absent in one-dimensional case~\cite{DN}. In fact, its emergence originates from the geometric term $\frac{d(N-1)}{r}u_r$ in the radial Laplacian: the pair $(q_*(c_*t-r),\, c_*t)$ is no longer a particular solution of \eqref{eq1.1}.
Our shift coefficient differs from that of the Stefan problem in \cite{DHZ}, owing to the different boundary conditions and semi-wave profiles. Proposition \ref{pr1.5}, proved in Section~4.1, shows that its limit as $\delta\downarrow0$ equals $d(N-1)/c_0$. This is the geometric logarithmic coefficient in the pushed Cauchy problem \eqref{eq1.6}. Interestingly, this limit holds without assuming that $f$ is of pushed type.

\end{remark}
	
The rest of this paper is organized as follows. In Section~2, we present some important lemmas, including the comparison principle, and a priori estimates of solutions, and then obtain the global existence of the solution. Section~3 establishes the trichotomy theorem via upper and lower solutions and mass estimates, and derives the precise shrinking rate for $\delta>1$. In Section~4, by applying a suitable perturbation to the semi-wave solution in Proposition \ref{pr1.3}, we obtain the precise propagation profile when spreading occurs. In Section~4.1, we describe how the constant \(c_N(\delta)\) in the logarithmic shifting term is defined and prove its limit as $\delta\downarrow0$. Sections~4.2 and~4.3 provide bounds for \(h(t)-c_*t+c_N(\delta)\log t\), and in Section~4.4 we obtain the convergence result.	In Section 5, we prove the local existence and uniqueness.

	\section{Basic results}  
	
	\subsection{Local existence and uniqueness and comparison principles}
	\begin{theorem}\label{th2.1}
		Suppose that $f\in C^1$ with $f(0)=0$, $N\geq 2$, and $\delta>0$. Then for any given $u_0 \in \mathcal{X}(h_0)$ and any $\alpha \in (0, 1)$, there exists $T > 0$ such that problem \eqref{eq1.1} admits a unique solution
		\[
		(u(t,r),  h(t)) \in C^{\frac{1+\alpha}{2}, 1+\alpha}(\bar{\Omega}_T) \times C^{1+\frac{\alpha}{2}}([0, T])
	 \mbox{ with }
		\Omega_T := \{(t, r) \in \mathbb{R}^2 : t \in (0, T], \, r \in [0, h(t)]\}.
		\]
		Moreover,
		\[
		\begin{cases}
			\|u(t,r)\|_{C^{\frac{1+\alpha}{2}, 1+\alpha}(\bar{\Omega}_T)} +  \|h(t)\|_{C^{1+\frac{\alpha}{2}}([0, T])} \leq C, \\[6pt]
			h(t) \in C^{1+\frac{1+\alpha}{2}}((0, T]), \quad u(t,r) \in C^{1+\frac{\alpha}{2}, 2+\alpha}(\Omega_T), \quad u(t,r) > 0 \text{ in } \Omega_T,
		\end{cases}
		\]
		Here, $C$ and $T$ depend only on $h_0$, $N$, and $\|u_0\|_{C^2([0, h_0])}$.
	\end{theorem}
To make the paper read more smoothly, we put the proof  in Section 5.

In the following, we present two comparison principle lemmas. Since the proofs are similar to those in \cite{DN}, we omit the details here.
	\begin{lemma}\label{le2.2}
		Suppose that $f$ is $C^1$ with $f(0) = 0$,  $N\geq 2$, $\delta>0$,  $T \in (0, \infty)$, $\bar{g}(t), \bar{h}(t) \in C^1([0, T])$, and $\bar{u}(t,r) \in C(\bar{D}_T) \cap C^{1,2}(D_T)$, where 
		\[
		D_T := \{(t, r) \in \mathbb{R}^2 : 0 < t \leq T, \, \bar{g}(t) < r < \bar{h}(t)\},
		\]
		and $0\leq \bar{g}(t) \leq h(t) $ for $t\in[0, 	T]$. Assume further that
		\[
		\begin{cases}
			\bar{u}_t(t,r) - d (\bar{u}_{rr}(t,r)+\frac{N-1}{r}\bar{u}_r(t,r)) \geq f(\bar{u}(t,r)), & 0 < t \leq T, \, {\bar{g}(t)} < r < \bar{h}(t), \\
			\bar{u}(t,\bar{g}(t)) \geq u(t,\bar{g}(t)), & 0 \leq t \leq T, \\
			\bar{u}(t,\bar{h}(t)) = \delta, \, \bar{h}'(t) \geq -\frac{d}{\delta} \bar{u}_r(t,\bar{h}(t)), & 0 < t \leq T, \\
			\bar{u}(t,h(t)) \geq \delta \quad\text{if } h(t) \leq \bar{h}(t) , & \\
			\bar{h}(0)>h_0, \,  \bar{u}(0, r)\geq\not\equiv u_0(r), & r \in [\bar{g}(0), h_0],
		\end{cases}
		\]
		where $(u, h)$ is the solution to \eqref{eq1.1}. Then, for $t\in(0, T]$ we have 
			\[	\bar{h}(t) > h(t) , \ \mbox{ and }\
			\bar{u}(t, r) > u(t, r) \ \mbox{ for }\ \bar{g}(t) < r < h(t).\]
	
	\end{lemma}

	\begin{lemma}\label{le2.3}
		Suppose that $f$ is $C^1$ with $f(0) = 0$,  $N\geq 2$, $\delta>0$, $T \in (0, \infty)$, $ \bar{h}(t) \in C^1([0, T])$, and $\bar{u}(t,r) \in C(\bar{D}_T) \cap C^{1,2}(D_T)$, where 
		\[
		D_T := \{(t, r) \in \mathbb{R}^2 : 0 < t \leq T, \, 0 < r < \bar{h}(t)\}.
		\]
		Assume further that
		\[
		\begin{cases}
			\bar{u}_t(t,r) - d (\bar{u}_{rr}(t,r)+\frac{N-1}{r} \bar{u}_r(t,r) )\geq f(\bar{u}(t,r)), & 0 < t \leq T, \, 0 < r < \bar{h}(t), \\
			\bar{u}_r(t,0) =0, \, 	\bar{u}(t,\bar{h}(t)) = \delta, & 0 < t \leq T,  \\
		  \bar{h}'(t) \geq -\frac{d}{\delta} \bar{u}_r(t,\bar{h}(t)), & 0 < t \leq T,  \\
			\bar{u}(t,h(t)) \geq \delta \quad\text{if } h(t) \leq \bar{h}(t), & \\
	\bar{h}(0)>h_0, \,  \bar{u}(0, r)\geq \not\equiv u_0(r), & r \in [0, h_0],
		\end{cases}
		\]
		where $(u,  h)$ is the solution to \eqref{eq1.1}. Then,  for $t\in(0, T]$ we have 
		\[
	h(t)< \bar{h}(t) \ \mbox{ and }\
		u(t, r) < \bar{u}(t, r) \  \mbox{ for }\  0 < r < h(t).
		\]
	\end{lemma}
	The  pair $(\bar{u}, \bar{h})$ in Lemmas \ref{le2.2} and \ref{le2.3} is usually referred to as an upper solution of the problem. A lower solution can be defined by reversing all inequalities in the appropriate places. There is also a symmetric version of Lemma \ref{le2.2}, where the conditions on the left and right boundaries are interchanged. Corresponding comparison results hold for lower solutions in each case.

	\subsection{ 
		A priori bounds and global existence of the solution
	}

	To analyze the global existence of solutions, we define $	E(t):=\int_{0}^{h(t)} r^{N-1} u(t, r) d r$.
	\begin{lemma}\label{le2.4}
		Suppose that $(\mathbf{f_m})$ holds, $N\geq 2$, $\delta>0$, and $T \in(0, \infty)$. If $(u,  h)$ is a solution to \eqref{eq1.1} for $0 < t < T$, then there exist  positive constants $C_{0}$  and $C_1$, independent of $T$, such that  
		\[
		C_{0} \leq u(t, r) \leq C_{1} \quad \text{for } t \in (0, T) \text{ and } r \in [0, h(t)].
		\]
		Moreover, we have
	$\inf_{t\in[0, T)} h(t)>0$.
	\end{lemma}
	\begin{proof}
		Clearly, the constant functions $\underline{u}(t, r) \equiv C_0:=\min\{1/2, \min_{r\in [0,h_{0}]} u_0(r)\}$ and $\bar{u}(t, r) \equiv C_1:=\|u_{0}\|_{L^{\infty}([0, h_0])} + 1$ serve as a lower and an upper solution, respectively, for problem \eqref{eq1.1} over the set $\{(t, r): t \in [0, T), r \in [0, h(t)]\}$. By the  comparison principle, we thus obtain the uniform bound
		\[
		C_0 \leq u(t, r) \leq C_{1} \quad \text{for } t \in [0, T),\, r \in [0, h(t)].
		\]
		This completes the proof of the first part.

By $(\mathbf{f_m})$, there exists $M>0$ such that $f(u)\geq -Mu$ for $u\in[0, C_1]$.		Direct calculation gives 		
\begin{align*}
	E'(t) &= \int_0^{h} r^{N-1}u_t\,dr + h'(t)h^{N-1}u(t,h) \\
	&= \int_0^{h} r^{N-1}\Bigl[d\bigl(u_{rr}+\tfrac{N-1}{r}u_r\bigr)+f(u)\Bigr]dr - d\,h^{N-1}u_r \\
	&= \int_0^{h} \Bigl[d\bigl(r^{N-1}u_{rr}+(N-1)r^{N-2}u_r\bigr)+r^{N-1}f(u)\Bigr]dr - d\,h^{N-1}u_r \\
	&= \int_0^{h(t)} r^{N-1}f(u)\,dr\geq -M E(t).
\end{align*}
		It follows that
		\[
		E(t) \geq E(0)e^{-Mt} > 0 \quad \text{for } t \in (0, T).
		\]
		Hence,
			$
		\inf_{t \in[0, T)} E(t) \geq E(0)e^{-MT} > 0$.
		Clearly this implies
	$
		\inf_{t \in[0, T)}h(t) > 0$.
	\end{proof}

	\begin{lemma}\label{le2.5}
		Suppose that $(\mathbf{f_m})$ holds,  $N\geq 2$, $\delta>0$,  $(u, h)$  is a solution to \eqref{eq1.1} defined for  $t \in[0, T)$  with $T \in\left(0, \infty\right) $. Then there exists  $C_2>0$ independent of $T$ such that
		\begin{equation}
		 \left|h^{\prime}(t)\right| \leq C_{2} \,\text{ for } \,t\in [0, T).
		\end{equation}
	\end{lemma}

		\begin{proof}
		{\bf{Step 1.}}
		We  construct a lower solution to obtain a lower bound of $h'(t)$.

		Define
		\[
		\underline{u}(t,r)=c\delta\left[e^{-\frac{(h(t)-r)m}{c}}-1\right]+\delta,\quad t\geq 0,\, r\in[h(t)-k/m,h(t)],
		\]
		where the constants $0<c<1$, $k>0$, and $m>0$ will be specified later.

		 Let $e_0=\min\{\min_{r\in[0,h_0]}u_0(r),\frac{\delta}{2},\frac12\}$ and  choose $c\in(0,1)$ sufficiently close to $1$ such that
	\[
	k:=-c\ln\left(\frac{e_0}{c\delta}+1-\frac{1}{c}\right)>0\ \mbox{ with }\ 	\underline{u}(t,h(t)-k/m)=e_0.
	\]	
	Choose 
	\[m > \max\left\{ 
	\frac{u_0'(h_0)}{\delta}, \frac{e^{k/c}\max_{r\in [0,h_0]}|u_0'(r)|}{\delta},
\sqrt{\dfrac{c e^{k/c} \max_{u\in[0,\delta]} |f(u)|}{d\delta (1-c)}}  
	\right\}\]
	 and define
\[
T_1=\sup\left\{t\leq T: h'(s)>-\frac{d}{\delta}\underline{u}_r(s,h)=-dm \ \text{for all } s\in [0,t)\right\}.
\]
It is easy to see  that $0<T_1\leq T$ is well defined.	
			We claim that $T_1=T$. Suppose for contradiction that
			$T_1<T$, then 
		\begin{equation}\label{eq2.2}
			h'(t)>-dm \ \mbox{ for } t\in[0, T_1) \ \mbox{ and } h'(T_1)=-dm.
		\end{equation}
		
In the following, we will prove that 		
		\begin{equation}\label{eq2.3}
			\begin{cases}
				\underline{u}_t-d\underline{u}_{rr}-d\frac{N-1}{r}\underline{u}_r-f(\underline{u})\leq0,&t\in(0,T_1],\,g(t)<r<h(t),\\
					\mathcal{B}\underline{u}(t,g(t))\leq	\mathcal{B}u(t,g(t)),\, \underline{u}(t,h(t))=\delta,&t\in(0,T_1],\\
				\underline{u}(0,r)\leq u_0(r), &r\in[g(0),h_0],
			\end{cases}
		\end{equation}
		where 
				\[
		\mathcal{B}u:=a(t)u-(1-a(t))u_r,\quad g(t):=\max \{h(t)-k/m,0\}, \quad 
		a(t):=
		\begin{cases}
			1, & h(t)\ge k/m,\\[6pt]
			0, & h(t)<k/m.
		\end{cases}
		\]	
 Then we can use the strong comparison principle and Hopf boundary lemma to obtain
	\[
h'(T_1)=-\frac{d}{\delta}u_r(T_1,h(T_1))>-\frac{d}{\delta}\underline{u}_r(T_1,h(T_1))=-dm,
\]
which contradicts \eqref{eq2.2}, proving the claim.

If $h(t)\geq k/m$, then   for $t\in[0, T_1]$,
\[ \mathcal{B} \underline{u}(t, g(t))=\underline{u}(t,h(t)-k/m)=e_0\leq u(t,h(t)-k/m)=\mathcal{B} u(t, g(t)).
		\]
	If $h(t)<k/m$, then for $t\in[0, T_1]$,
	\[\mathcal{B} \underline{u}(t, g(t))=-\underline{u}_r(t,0)<0=-u_r(t,0)=\mathcal{B} u(t, g(t)).\]
	By the definition of $\underline{u}$, we immediately obtain
		\[\underline{u}(t,h(t))=\delta,\ \mbox{ and }\
		\underline{u}(0,r)<u_0(r),\ \mbox{ for }\ r\in [g(0),h_0),
		\]
		which proves the third condition in \eqref{eq2.3}.

		By direct calculation, we have
		\begin{equation}\label{eq2.4}
			\underline{u}_t=-\delta m e^{-\frac{(h-r)m}{c}}h',\qquad \underline{u}_r = \delta m e^{-\frac{(h-r)m}{c}}>0, \qquad \underline{u}_{rr} = \frac{\delta m^2}{c}e^{-\frac{(h-r)m}{c}}.
		\end{equation}
		Since $e_0\leq\underline{u}\leq\delta$, by \eqref{eq2.2}, \eqref{eq2.4} and the choice of $m$, for $g(t)<r<h(t)$ we have
		\[
		\underline{u}_t-d\underline{u}_{rr}-d\frac{N-1}{r}\underline{u}_r-f(\underline{u})
		\leq -\frac{d\delta m^2(1-c)}{c}e^{-k/c}+\max_{v\in[0,\delta]}|f(v)|<0.
		\]
Therefore, we obtain the  desired conclusion
\begin{equation}
	h'(t)\geq -dm \ \mbox{ for }\ t\in[0, T).
\end{equation}
		
		\textbf{Step 2.} 	We  construct an upper solution to obtain an upper bound of $h'(t)$

For $ t>0$ and $h(t)-M^{-1}<r<h(t)$,  define		
		\[	\overline{u}(t,r):=(2C_1-\delta)[2M(h(t)-r)-M^2(h(t)-r)^2]+\delta,  \]
		where  $M$ is a constant satisfying
		\[M>\max \left\{\sqrt{\frac{LC_1}{d(2C_1-\delta)}}+m, \frac{\max_{r\in [0,h_0]}|u_0'(r)|}{2C_1-\delta}\right\}\]
	with $C_1$ given in Lemma \ref{le2.4} and $L=\max_{u \in [0,2C_1]}f'(u)$.
	In the following, we will prove that 		
	\begin{equation}
		\begin{cases}
		\overline{u}_t-d\overline{u}_{rr}-d\frac{N-1}{r}\overline{u}_r-f(\overline{u})\geq0,&t\in(0,T),\,\hat{g}(t)<r<h(t),\\
			\hat{\mathcal{B}}\overline{u}(t,\hat{g}(t))\geq	\hat{\mathcal{B}}u(t,\hat{g}(t)),\, \overline{u}(t,h(t))=\delta,&t\in(0,T),\\
			\overline{u}(0,r)\geq u_0(r), &r\in[\hat{g}(0),h_0],
		\end{cases}
	\end{equation}
	where 
	\[
	\hat{\mathcal{B}}u:=b(t)u-(1-b(t))u_r,\quad \hat{g}(t):=\max \{h(t)-1/M,0\}, \quad 
	b(t):=
	\begin{cases}
		1, & h(t)\ge 1/M,\\[6pt]
		0, & h(t)<1/M.
	\end{cases}
	\]		
	 Then, combining the free boundary condition, we can use the comparison principle and Hopf boundary lemma to obtain
	\[
h'(t)=-\frac{d}{\delta}u_r(t,h(t)) \leq 2M(2C_1-\delta)\frac{d}{\delta}\ \text{for } t\in (0,T).
\]
	implying the desired conclusion.

	By the definition of $\overline{u}$, it is easy to verify
			\[
		\overline{u}(0,r)\geq u_0(r)\ \text{for } r\in [\hat{g}(0),h_0], \ \mbox{ and }\ \overline{u}(t,h(t))=\delta
		\]
If $h(t)\geq1/M$, then   for $t\in[0, T)$
\[ \hat{\mathcal{B}} \overline{u}(t, \hat{g}(t))=\overline{u}(t,h(t)-1/M)=2C_1\geq u(t,h(t)-1/M)=\hat{\mathcal{B}} u(t, \hat{g}(t)).
\]
If $h(t)<1/M$, then for $t\in[0, T)$,
\[\hat{\mathcal{B}} \overline{u}(t, \hat{g}(t))=-\overline{u}_r(t,0)>0=-u_r(t,0)=\hat{\mathcal{B}} u(t, \hat{g}(t)).\]
		Moreover, for $t\in(0, T)$ and $\hat{g}(t)<r<h(t)$, by the definition of $M$,  we obtain
		\begin{align*}
			\overline{u}_t-(d\overline{u}_{rr}+d\tfrac{N-1}{r}\overline{u}_r)-f(\overline{u})&=(2C_1-\delta)(h'(t)+d\tfrac{N-1}{r})\left[2M-2M^2(h(t)-r)\right]+d(2C_1-\delta)2M^2-f(\overline{u})\\
			&\geq-2Mdm(2C_1-\delta)+d(2C_1-\delta)2M^2-2LC_1\\
			&=2M[dM-dm](2C_1-\delta)-2LC_1\geq 0,
		\end{align*}
	where we have used the fact $h'(t)\geq -dm$ for $t\in [0,T)$.
		 This completes the proof of the lemma.
	\end{proof}

	\begin{proof}[Proof of Theorem \ref{th1.1}.]
			Lemma \ref{le2.5} implies that as long as $h(T)>0$ the local solution obtained in
		Theorem \ref{th2.1} can be extended. Let  $\left[0, \,T_{\mathrm{max}}\right) $ be the maximal interval of existence obtained via this extension.
		Arguing indirectly, assume $T_{\max} < \infty$. Then, by Lemmas \ref{le2.4} and \ref{le2.5}, there exist constants $C_1, C_2$ such that for $t \in [0, T_{\max})$ and $r \in [0, h(t)]$,
		\[
		0 \leq u(t, r) \leq C_1, \quad |h'(t)|  \leq C_2, \quad |h(t)| \leq C_2 T_{\max} + h_0, \quad \inf_{t \in[0, T_{\max})}h(t)>0.
		\]
		Define
		$
	D_T := \{(t, r) \in \mathbb{R}^2 : t \in (0, T], \, r \in [0, h(t)]\}$.
		For any small constant $\varepsilon > 0$, we infer from Theorem \ref{th2.1}, Lemma \ref{le2.4}, and Lemma \ref{le2.5} that $u \in C^{\frac{1+\alpha}{2}, 1+\alpha}(\bar{D}_{T_{\max} - \varepsilon})$. By Schauder  estimates, for fixed $0 < T < T_{\max} - \varepsilon$, we have
		$
		\|u(t,r)\|_{C^{1+\frac{\alpha}{2}, 2+\alpha}(D_{T_{\max} - \varepsilon} \setminus D_T)} \leq M$,
		where $M$ depends on $T$, $T_{\max}$, and $C_i$ ($i = 1, 2$), but is independent of $\varepsilon$. Since $\varepsilon > 0$ can be arbitrarily small, it follows that for any $t \in [T, T_{\max})$,
	$	\|u(t, \cdot)\|_{C^{2+\alpha}([0, h(t)])} \leq M$.
		Now we can repeat the proof of Theorem \ref{th2.1} to conclude that there exists $\tau > 0$, depending on $M$ and $C_i$ ($i = 1, 2$), such that \eqref{eq1.1} with initial time $T_{\max} - \frac{\tau}{2}$ has a unique solution $(u(t,r), h(t))$ defined for larger $t$, at least up to $T_{\max} - \frac{\tau}{2} + \tau$. This contradicts the definition of $T_{\max}$. Thus, we conclude that $T_{\max} = \infty$.

	\end{proof}
	
\section{Long-time behavior}
In this section, by constructing upper and lower solutions and employing mass estimates, we establish the existence of the limit of the free boundary and derive a trichotomy result.

\subsection{Longtime dynamics for $0<\delta<1$}

\begin{lemma}\label{le3.1}
	Suppose that $(\mathbf{f_m})$ holds, $N \geq 2$, $0<\delta<1$, and $(u, h)$ is a solution to \eqref{eq1.1} for $t \in(0, \infty)$. Then there exists a constant $T_{0} \geq 0$  such that
	\[
	h'(t) \geq 0, \ \mbox{ and }\	u(t, r) \geq \delta,\quad \mbox{ for } t\geq T_{0},\ 0 \leq r \leq h(t).
	\]

\end{lemma}

\begin{proof}
 We consider the  problem
$	\underline{u}'(t)=f(\underline{u}(t))$ with initial value $\underline{u}(0)=\displaystyle\min_{0 \leq r \leq h(0)} u_0(r)\in(0,\delta]$.
	By $(\mathbf{f_m})$, $\underline{u}(t)$ is strictly increasing to $1$. Consequently, there exists a unique constant $T_{0} \geq 0$ such that $\underline{u}(T_{0})=\delta$ and $\underline{u}(0) \leq \underline{u}(t) \leq\delta$ for all $t \in[0, T_{0}]$.
	
	Define the region $\Omega_1=\{(t,r)\mid 0 \leq t \leq T_{0},\ 0 \leq r \leq h(t)\}$. By the parabolic comparison principle, we have $u(t, r) \geq \underline{u}(t)$ for all $(t, r) \in \Omega_1$. In particular, $u(T_{0}, r) \geq\delta$ for all $0 \leq r \leq h(T_{0})$. 
	  It is easy to verify $\tilde{u} \equiv \delta$ is a strict subsolution of \eqref{eq1.1} for $t\geq T_0$. 
	  Then combining this  with the free boundary condition, we complete the proof.

\end{proof}

\begin{lemma}\label{le3.2}
	Suppose that $(\mathbf{f_m})$ holds, $N \geq 2$, $0<\delta<1$, and $(u, h)$ is a solution to \eqref{eq1.1} for $t \in(0, \infty)$. Then
	\[
	\lim_{t \to \infty} h(t)=+\infty \quad \text{and} \quad \lim_{t \to \infty} u(t, r)=1 \quad \text{locally uniformly for } r \geq 0.
	\]
\end{lemma}

\begin{proof}
	\textbf{Step 1:} We prove that $\lim_{t\to \infty } h(t)=\infty$.

	By Lemma \ref{le3.1}, the limit  
	$
	h_{\infty}:=\lim_{t \to \infty} h(t)
	$
	exists and satisfies $h_{\infty} \in (0, +\infty]$. We argue by contradiction and suppose that $h_\infty \in (0, \infty)$. Define
	\begin{equation}
		s=\frac{r}{h(t)},\quad U(t, s)=u(t, r).
	\end{equation}
	Then we have the system
	\begin{equation}\label{eq3.2}
		\begin{cases}
			U_{t}-\dfrac{d}{h^2(t)} \Delta_{s} U-\dfrac{h'(t)}{h(t)} s U_{s}=f(U), & t>0,\ s \in[0,1], \\
			U_{s}(t, 0)=0,\quad U(t, 1)=\delta, & t>0, \\
			h'(t)=-\dfrac{d}{\delta h(t)} U_{s}(t, 1), & t>0,
		\end{cases}
	\end{equation}
	where $\Delta_{s} U=U_{ss}+\dfrac{N-1}{s}U_{s}$.
	Let $B_1:=\{x\in\mathbb{R}^N:|x|<1\}$ and identify $U$ with its radial extension to $B_1$. Then \eqref{eq3.2} gives
	\[
	\begin{cases}
	U_t-\dfrac{d}{h^2(t)}\Delta U-\dfrac{h'(t)}{h(t)}x\cdot\nabla U=f(U), & t>0,\ x\in B_1,\\
	U=\delta, & t>0,\ x\in\partial B_1.
	\end{cases}
	\]
	By Lemmas \ref{le2.4} and \ref{le2.5}, the standard $L^p$ estimates yield, for any $p>1$,
	\[
	\|U\|_{W_p^{2,1}((n,n+2)\times B_1)}\leq C_p,\quad n\geq1.
	\]
	Here $C_p$ is independent of $n$.
	For $p$ sufficiently large, the Sobolev embedding theorem yields
	\[
	\| U\| _{C^{\frac{1+\alpha}{2}, 1+\alpha}([1, \infty) \times [0,1])} \leq C_{p}'<\infty
	\]
	for some exponent $\alpha \in(0,1)$. Together with the free boundary condition, this implies that $h'(t)$ is uniformly continuous on $[1, \infty)$ and hence
	$
	\lim_{t \to \infty} h'(t)=0$.

Let $\{t_n\}_{n=1}^{\infty}$ be any sequence such that $t_n \to \infty$, and define
$
U_n(t, s):=U(t_n+t, s)$.
After passing to a subsequence, we have
\[
U_n \to \widetilde{U} \quad \text{in } C_{\mathrm{loc}}^{\frac{1+\alpha}{2}, 1+\alpha}(\mathbb{R} \times [0,1]),
\]
where $\widetilde{U}$ satisfies
\begin{equation}\label{eq3.3}
	\begin{cases}
		\widetilde{U}_{t}-\dfrac{d}{h_{\infty}^2} \Delta_{s} \widetilde{U}=f(\widetilde{U}),\quad \widetilde{U}\geq \delta, & t \in \mathbb{R},\ s \in[0,1], \\
		\widetilde{U}_{s}(t, 0)=0,\quad \widetilde{U}(t, 1)=\delta, & t \in \mathbb{R}, \\
		\widetilde{U}_{s}(t, 1)=0, & t \in \mathbb{R}.
	\end{cases}
\end{equation}

Since $\hat{U}\equiv\delta$ is a strict subsolution of \eqref{eq3.3}, by the strong maximum principle and Hopf boundary lemma, we obtain
$
\widetilde{U}_{s}(t, 1)<0$,
which contradicts the third equation in \eqref{eq3.3}. Therefore, we conclude that
$
\lim_{t \to \infty} h(t)=+\infty $.

	\textbf{Step 2:} We prove  $\lim_{t \to \infty} u(t, r)=1 $ locally uniformly for $r \geq 0$.

For fixed $L>0$, consider the auxiliary parabolic problem
\begin{equation}
	\begin{cases}
		\partial_{t} v_{L}-d \Delta_{r} v_{L}=f(v_{L}), & t>0,\ 0 \leq r \leq L, \\
		\partial_{r} v_{L}(t, 0)=0,\quad v_{L}(t, L)=\delta, & t>0, \\
		v_{L}\left(0, r\right)=\delta, & 0 \leq r \leq L .
	\end{cases}
\end{equation}

Since $\underline{v}\equiv\delta$ and $\bar{v}\equiv1$ serve as lower and upper solutions, respectively, $v_L(t,r)$ is nondecreasing in $t$. Hence the limit
$
v_{L}^{\infty}(r):=\lim_{t \to \infty} v_{L}(t,r)
$
exists and satisfies $\delta \leq v_{L}^{\infty}(r) \leq 1$ for $r \in[0, L]$.
By standard regularity theory, $v_L^\infty$ satisfies
\[
-d\Delta _{r}v_{L}^{\infty }=f(v_{L}^{\infty }),\quad 0\leq r\leq L,\quad \partial _{r}v_{L}^{\infty }(0)=0,\quad v_{L}^{\infty }(L)=\delta .
\]

For any $L_1>L$, the restriction of $v_{L_1}$ to $[0, \infty) \times [0, L]$ is an upper solution for $v_L$, and hence $v_L(t, r) \leq v_{L_1}(t, r)$ on that domain. Consequently, $v_L^\infty(r) \leq v_{L_1}^\infty(r)$ on $[0, L]$, and thus
$
v_{\infty}(r):=\lim_{L \to \infty }v_{L}^{\infty }(r)
$
exists for $r \geq 0$.
By elliptic regularity, $v_\infty$ satisfies
\[
-d\Delta _{r}v_{\infty }=f(v_{\infty }),\quad r\geq 0,\quad \partial _{r}v_{\infty }(0)=0,\quad \delta \leq v_{\infty }\leq 1.
\]

Now consider the problem
$
z'(t)=f(z(t))$ for $t>0$ with initial value $z(0)=\delta$.
It follows that $\delta \leq z(t) \leq 1$ for all $t \geq 0$ and $\lim\limits_{t \to \infty} z(t)=1$. By comparison, $v_{\infty}(r) \geq 1$ for all $r \geq 0$. Since $v_\infty \le 1$, we deduce that $v_{\infty} \equiv 1$.

By Lemma \ref{le3.1} and Step 1, there exists $T_L>0$ such that
\[
h(t)>L \quad \text{and} \quad u(t,r)\geq \delta \quad \text{for } t\geq T_L,\ r\in[0,h(t)].
\]
By the comparison principle, $v_L(t, r) \leq u(t+T_L, r)$ on $[0, \infty) \times [0, L]$. Hence
\[
\liminf_{t\to \infty} u(t, r) \geq v_{L}^{\infty }(r),\quad \forall\, r \in[0, L].
\]
Letting $L \to \infty$, we obtain
\[
\liminf_{t \to \infty} u(t, r) \geq 1,\quad \text{locally uniformly for } r \geq 0.
\]

Let $w^{*}(t)$ solve the problem
$
\left(w^{*}\right)^{\prime }(t)=f(w^{*}(t))$ for $t>0$ with initial value $w^{*}(0)=\max _{0 \leq r \leq h(0)} u(0, r)$.
Under the assumptions on $f$, we have $\lim_{t \to \infty}w^{*}(t)=1$ and $w^{*}(t)\geq \delta$ for $t\geq0$.
By the comparison principle,
\[
u(t,r)\leq w^{*}(t),\quad \text{for } t\geq 0,\ 0\leq r\leq h(t).
\]
Taking $t \to \infty$ gives
\[
\limsup_{t \to \infty} u(t, r) \leq 1,\quad \text{locally uniformly for } r \geq 0.
\]
Combining the above estimates, we conclude that
\[
\lim_{t \to \infty} u(t, r)=1,\quad \text{locally uniformly for } r \geq 0 .
\]

\end{proof}

\subsection{Longtime dynamics for $\delta=1$}

\begin{lemma}
	Suppose that $(\mathbf{f_m})$ holds, $N \geq 2$, $\delta=1$, and $(u, h)$ is a solution to \eqref{eq1.1} for $t \in(0, \infty)$. Then
	\[
	\lim_{t \to \infty} h(t)=h_{\infty} \in(0, \infty) \quad \text{and} \quad \lim_{t \to \infty} u(t, r)=1 \quad \text{uniformly for } r \in[0, h(t)].
	\]
\end{lemma}

\begin{proof}
\textbf{Step 1.} We prove that $u(t,r)\to1$ uniformly in $[0, h(t)]$ as $t\to\infty$.

Consider the two initial-value problems 
\[
\left\{
\begin{aligned}
	w_i'(t)&=f(w_i(t)), \ \mbox{ for }\ t>0, \quad i=1,2,\\
	w_1(0)&=\max_{0\le r\le h(0)}u_0(r), \quad
	w_2(0)=\min_{0\le r\le h(0)}u_0(r).
\end{aligned}
\right.
\]
By the assumptions on $f$, the corresponding solutions satisfy $w_1(t)\ge 1$, $w_2(t)\le 1$, and
$
\lim_{t \to \infty} w_{1}(t)=\lim_{t \to \infty} w_{2}(t)=1 $.
By the comparison principle, we have $
w_{2}(t)\leq u(t,r)\leq w_{1}(t)$ for all $t>0$ and $0 \leq r \leq h(t)$.
Letting $t \to \infty$, we obtain the uniform convergence
\[
\lim_{t \to \infty}\sup_{0 \leq r \leq h(t)}|u(t, r)-1|=0 .
\]

	\textbf{Step 2.} We prove that $h(t)$ converges to a finite positive limit.
	
	Since $f'(1)<0$, the solutions $w_1$ and $w_2$ in Step 1 converge exponentially to $1$. Thus there exist $C,\sigma>0$ such that
	\[
	\sup_{0\leq r\leq h(t)}|u(t,r)-1|\leq Ce^{-\sigma t},\quad t\geq0.
	\]
	Recall that $E(t)=\int_0^{h(t)}r^{N-1}u(t,r)\,dr$. By Lemma \ref{le2.4}, $E(t)\geq C_0h^N(t)/N$ and
	\[
	E'(t)=\int_0^{h(t)}r^{N-1}f(u(t,r))\,dr.
	\]
	Since $f\in C^1$ and $f(1)=0$, increasing $C$ if necessary, we obtain
	\[
	|E'(t)|\leq Ce^{-\sigma t}\frac{h^N(t)}{N}\leq\frac{C}{C_0}e^{-\sigma t}E(t).
	\]
	Hence $| (\log E(t))' |\leq (C/C_0)e^{-\sigma t}$, so $\log E(t)$ converges to a finite limit. Consequently, $E(t)\to E_\infty\in(0,\infty)$. Moreover,
	\[
	\left|\frac{NE(t)}{h^N(t)}-1\right|\leq\sup_{0\leq r\leq h(t)}|u(t,r)-1|\to0,
	\]
	which implies $h(t)\to(NE_\infty)^{1/N}\in(0,\infty)$. This completes the proof.
\end{proof}

\subsection{Longtime dynamics for $\delta>1$}

\begin{lemma}
	Suppose that $(\mathbf{f_m})$ holds, $N \geq 2$, $\delta>1$, and $(u, h)$ is a solution to \eqref{eq1.1} for $t \in(0, \infty)$. Then
	\[
	\lim_{t \to \infty} h(t)=0 \quad \text{and} \quad \lim_{t\to \infty } u(t, r)=\delta \quad \text{uniformly for } r\in [0, h(t)].
	\]
Moreover, there exists $H_\infty>0$ such that
\[
h(t)=H_\infty e^{-\kappa t}\bigl(1+O(e^{-\kappa t})\bigr),\quad
\kappa=-\frac{f(\delta)}{N\delta}>0,
\ \mbox{ as } t\to\infty.
\]
\end{lemma}

\begin{proof}

Consider the  problem
$	\bar{u}'(t)=f(\bar{u}(t))$ with initial value $\bar{u}(0)=\max_{0 \leq r \leq h(0)} u_0(r)\in[\delta,\infty)$.
By $\mathbf{(f_m)}$, $\bar{u}(t)$ is strictly decreasing to $1$. Consequently, there exists a unique constant $T_{0} \geq 0$ such that $\bar{u}(T_{0})=\delta$ and $\bar{u}(t) \geq\delta$ for all $t \in[0, T_{0}]$.
 By the  comparison principle, we have $u(t, r) \leq \bar{u}(t)$ for all $t\in[0, T_0]$. In particular, $u(T_{0}, r) \leq\delta$ for all $0 \leq r \leq h(T_{0})$. 
It is easy to verify $\tilde{u} \equiv \delta$ is a strict upper solution of \eqref{eq1.1} for $t\geq T_0$. 
Then combining this conclusion with the free boundary condition, we have
\[h'(t)\leq 0 \ \mbox{ for }\ t\geq T_0, \mbox{ and }
	h_{\infty}:=\lim_{t \to \infty }h(t) \leq h(T_{0 }).
	\]

	We now prove by contradiction that $h_{\infty}=0$. Suppose on the contrary that $0<h_{\infty}<+\infty$.
	Let $t_n\to\infty$ and define
	\[
	U(t,s)=u(t,h(t)s),\quad U_n(t,s)=U(t+t_n,s),\quad s\in[0,1].
	\]
	As in the proof of Lemma \ref{le3.2}, the $L^p$ estimates imply $h'(t)\to0$ and yield, along a subsequence, local convergence of $U_n$ to $\widetilde{U}$ satisfying
	\begin{equation}
		\begin{cases}
			\widetilde{U}_{t}-\dfrac{d}{h_{\infty}^{2}} \Delta_{s} \widetilde{U}=f(\widetilde{U}), \quad 	\widetilde{U}\leq \delta, & t \in \mathbb{R},\ s \in[0,1], \\
			\widetilde{U}_{s}(t, 0)=0,\quad \widetilde{U}(t, 1)=\delta, & t \in \mathbb{R}, \\
			\widetilde{U}_{s}(t, 1)=0, & t \in \mathbb{R} .
		\end{cases}
	\end{equation}
	Since $W\equiv \delta$ is a strict upper solution,
the strong comparison principle and the Hopf boundary lemma yield
$	\widetilde{U}_{s}(t, 1)>0$, a contradiction. Therefore $h_\infty=0$.

Since $h(t)\to0$, there exists $S_0\geq T_0$ such that $h(t)<k/m$ for $t\geq S_0$. By the lower solution constructed in the proof of Lemma \ref{le2.5} and the bound $u(t,r)\leq\delta$, we have
\[
0\leq\delta-u(t,r)\leq c\delta\left[1-e^{-\frac{(h(t)-r)m}{c}}\right]\leq\delta m h(t),\quad t\geq S_0,\ 0\leq r\leq h(t).
\]
Letting $t\to\infty$, we conclude that
\[
\lim_{t\to\infty}\sup_{0\leq r\leq h(t)}|u(t,r)-\delta|=0.
\]

 Recall from the proof of Lemma \ref{le2.4} that
\[
 E(t)=\int_0^{h(t)}r^{N-1}u(t,r)\,dr \ \mbox{ and }\ E'(t)=\int_0^{h(t)}r^{N-1}f(u(t,r))\,dr.
\]
 Since $h(t)\to0$, we may choose $T\geq S_0$ such that $mh(t)\leq1/2$ for $t\geq T$. Then $\delta/2\leq u(t,r)\leq\delta$, and integration gives
\begin{equation}\label{eq3.6}
	\left\{
	\begin{aligned}
		&\frac{\delta}{2N}h(t)^N\leq E(t)\leq\frac{\delta}{N}h(t)^N,\\
		&0\leq\frac{\delta}{N}h(t)^N-E(t)
		=\int_0^{h(t)}r^{N-1}(\delta-u(t,r))\,dr
		\leq\frac{\delta m}{N}h(t)^{N+1}.
	\end{aligned}
	\right.
\end{equation}

Since $f\in C^1$ and $f(\delta)+N\kappa\delta=0$, there exists a constant $C>0$ such that
\[
|f(v)+N\kappa v|\leq C|v-\delta|\quad\text{for }v\in[\delta/2,\delta].
\]
Hence, we have
\[
\begin{aligned}
|E'(t)+N\kappa E(t)|
&\leq\int_0^{h(t)}r^{N-1}|f(u(t,r))+N\kappa u(t,r)|\,dr\\
&\leq C\int_0^{h(t)}r^{N-1}(\delta-u(t,r))\,dr
\leq C h(t)^{N+1}.
\end{aligned}
\]
Dividing by the lower bound for $E(t)$ in \eqref{eq3.6}, and increasing $C$ and $T$ if necessary, yields
\begin{equation}\label{eq3.7}
\left|\frac{E'(t)}{E(t)}+N\kappa\right|\leq C h(t)\leq N\kappa/2,\qquad t\geq T.
\end{equation}
 It follows that $E'(t)/E(t)\leq-N\kappa/2$. Integrating this inequality from $T$ to $t$, and then using \eqref{eq3.6}, we have
\[
E(t)\leq E(T)e^{-N\kappa(t-T)/2},\qquad
h(t)\leq\left(\frac{2NE(T)}{\delta}\right)^{1/N}e^{-\kappa(t-T)/2}.
\]
In particular, $\int_T^\infty h(s)\,ds<\infty$. On the other hand,
\[
\log\bigl(E(t)e^{N\kappa t}\bigr)
=\log\bigl(E(T)e^{N\kappa T}\bigr)
+\int_T^t\left(\frac{E'(s)}{E(s)}+N\kappa\right)ds.
\]
By \eqref{eq3.7},  as $t\to\infty$, we can  obtain
\[
E(t)e^{N\kappa t}\to E_\infty\in(0,\infty).
\]
Moreover, the second inequality in \eqref{eq3.6} implies
\[
0\leq1-\frac{NE(t)}{\delta h(t)^N}\leq mh(t)\to0.
\]
Combining these two limits, we obtain
\[
h(t)^Ne^{N\kappa t}
=\frac{h(t)^N}{E(t)}\,E(t)e^{N\kappa t}
\longrightarrow\frac{NE_\infty}{\delta}.
\]
Consequently, $h(t)e^{\kappa t}\to H_\infty:=(NE_\infty/\delta)^{1/N}>0$ and $h(t)\leq Ce^{-\kappa t}$ for all sufficiently large $t$. 

Integrating \eqref{eq3.7} from $t$ to infinity, we obtain
\[
\left|\log\frac{E(t)e^{N\kappa t}}{E_\infty}\right|
\leq C\int_t^\infty h(s)\,ds\leq Ce^{-\kappa t}.
\]
Since $e^z=1+O(z)$ as $z\to0$, it follows that
\[
E(t)=E_\infty e^{-N\kappa t}\bigl(1+O(e^{-\kappa t})\bigr).
\]
Finally, \eqref{eq3.6} and $h(t)=O(e^{-\kappa t})$ yield
\[
h(t)^N=\frac{N}{\delta}E(t)\bigl(1+O(h(t))\bigr)
=H_\infty^N e^{-N\kappa t}\bigl(1+O(e^{-\kappa t})\bigr).
\]
Then, we conclude that
\[
h(t)=H_\infty e^{-\kappa t}\bigl(1+O(e^{-\kappa t})\bigr)
\quad\text{as }t\to\infty.
\]
This completes the proof.
\end{proof}

	\section{Precise propagation profile}
In this section, we prove Theorem \ref{th1.4} and obtain the logarithmic shifts by constructing precise upper and lower solutions. The method is inspired by \cite{DN,DHZ,ATH}.
	\subsection{Formula for $c_N(\delta)$}
	
Fix $\delta\in(0,1)$ and $\mu>0$, and consider the problem
\begin{equation}
	d q'' - c q' + f(q) = 0,\quad
	q(0)=\delta,\quad q'(0)=\frac{c}{\mu},\quad q(\infty)=1.
\end{equation}
By \cite{DL,DN}, for every $c\in[0,c_0]$, the system
\begin{equation}\label{eq4.2}
	P_c'(q) = \frac{c}{d} - \frac{f(q)}{d\,P_c(q)} \quad \text{in }  (0,1),\qquad
	P_c(1)=0,\quad P_c'(1)<0,
\end{equation}
admits a unique solution $P_c(q)$ satisfying
\[
P_c'(1) = \frac{c - \sqrt{c^2-4d f'(1)}}{2d},\qquad
P_c(q)>0 \quad \text{for } q\in(0,1).
\]
 The following lemma gives some useful properties of $P_c$.

\begin{lemma}\label{le4.1}
	For any $0\le c_1<c_2\le c_0$ and $\bar c\in[0,c_0]$,
	\[
	P_{c_1}(q) > P_{c_2}(q)\quad\text{in }[0,1),\qquad 
	\lim_{c\to\bar c} P_c(q) = P_{\bar c}(q)\quad\text{uniformly in }[0,1].
	\]
	Moreover, $P_{c_0}(0)=0$ and $P_{c_0}(q)>0$ for $q\in(0,1)$.
	The function $c\mapsto P_c(q)$ is $C^2$ on $(0,c_0)$.
\end{lemma}

\begin{proof}
The first part of the conclusion is provided by Lemma~6.1 of \cite{DL}.
 Fix $0<a<b<c_0$. By the monotonicity of $P_c$ and the formula for $P_c'(1)$, there exist $m,M>0$ such that
{
\begin{equation}\label{eq4.3}
m(1-q)\leq P_c(q)\leq M(1-q),\qquad c\in[a,b],\quad q\in[0,1].
\end{equation}
}
 For $c\in(a,b)$ and sufficiently small $h\neq0$, subtracting the equations in \eqref{eq4.2} for $P_{c+h}$ and $P_c$ and integrating, we obtain
\[
\frac{P_{c+h}(q)-P_c(q)}{h}
=-\frac{1}{d}\int_q^1
\exp\!\left(-\frac{1}{d}\int_q^\xi\frac{f(s)}{P_{c+h}(s)P_c(s)}\,ds\right)d\xi,
\]
 Letting $h\to0$, dominated convergence yields
\begin{equation}\label{eq4.4}
	\frac{d}{dc} P_c(q)
	= -\frac{1}{d}\int_q^1
	\exp\!\left(-\frac{1}{d}\int_q^\xi \frac{f(s)}{P_c(s)^2}\,ds\right)d\xi < 0,
	\qquad q\in[0,1).
\end{equation}
 In particular, $\left|\frac{d}{dc}P_c(q)\right|\leq(1-q)/d$. Since $f(s)=O(1-s)$, the preceding estimates imply
\[
\left|\frac{f(s)}{P_c(s)^3}\,\frac{d}{dc}P_c(s)\right|\leq\frac{C}{1-s},
\qquad
\int_q^1\log\frac{1-q}{1-\xi}\,d\xi=1-q,
\]
with $C$ independent of $c\in(a,b)$. Thus we may differentiate \eqref{eq4.4} under the integral sign to obtain
\begin{equation}
	\frac{d^2}{dc^2}P_c(q)
	= -\frac{2}{d^2}\int_q^1 \left[
	\exp\!\left(-\frac{1}{d}\int_q^\xi \frac{f(s)}{P_c(s)^2}\,ds\right)
	\int_q^\xi \frac{f(s)}{P_c(s)^3}\,\frac{d}{dc}P_c(s)\,ds
	\right] d\xi.
\end{equation}
 The same bounds and dominated convergence give continuity of both derivatives in $c$, uniformly for $q\in[0,1]$, and $\left|\frac{d^2}{dc^2}P_c(q)\right|\leq C(1-q)$. Since $a,b$ are arbitrary, the proof is complete.
\end{proof}

Define
\begin{equation}
	\eta(c,\mu) := P_c(\delta) - \frac{c}{\mu}, \qquad \mu_0 := \frac{d}{\delta}.
\end{equation}
 Since $P_{c_0}(0)=0$ and $P_{c_0}'(q)<c_0/d$ for $q\in(0,1)$, we obtain
\[
\eta(c_0,\mu_0) = P_{c_0}(\delta) - \frac{c_0}{\mu_0}
< \frac{c_0\delta}{d} - \frac{c_0}{\mu_0} = 0,
\qquad
\eta(0,\mu) = P_0(\delta) > 0.
\]
 By continuity, we may choose $\mu^*>\mu_0$ such that $\eta(c_0,\mu^*)<0$. Since $\partial_\mu\eta(c_0,\mu)=c_0/\mu^2>0$, it follows that
\[
\eta(c_0,\mu) < 0 \quad \text{for all } \mu\in(0,\mu^*].
\]

\begin{proposition}
	Assume that condition $(\mathbf{f_m})$ holds and $\delta\in(0,1)$. Then, for each $\mu\in(0,\mu^*]$, there exists a unique pair
	$(c,q)=\bigl(c_*(\mu), q_{c_*(\mu)}(\mu;\cdot)\bigr)$ with $c_*(\mu)>0$ satisfying
	\begin{equation}\label{eq4.7}
		\begin{cases}
			d q'' - c q' + f(q) = 0, & q'>0 \quad \text{in } (0,\infty), \\
			q(0)=\delta, \quad q(\infty)=1, \quad q'(0)=\dfrac{c}{\mu},
		\end{cases}
	\end{equation}
 Moreover, $c_*(\mu)\in(0,c_0)$ and $c_*(\mu)$ is a $C^2$ function. Furthermore, for  $\mu\in(0,\mu^*]$,
	\[
	c_*'(\mu)>0>\left(\frac{c_*(\mu)}{\mu}\right)'.
	\]
\end{proposition}

\begin{proof}
 Fix $\mu\in(0,\mu^*]$. We have
\[
\eta(0,\mu)=P_0(\delta)>0>\eta(c_0,\mu).
\]
Since $c\mapsto\eta(c,\mu)$ is continuous and strictly decreasing, there exists a unique $c_*(\mu)\in(0,c_0)$ such that
\[
\eta(c_*(\mu),\mu)=0.
\]
We note that $P_c$ is also defined for $c>c_0$, and $c\mapsto\eta(c,\mu)$ remains continuous and strictly decreasing on $[0,\infty)$; see \cite[Lemma~6.1 and its proof]{DL}.
 Let $q$ solve
\[
q'=P_{c_*(\mu)}(q),\qquad q(0)=\delta.
\]
Since $P_{c_*(\mu)}(s)>0$ for $\delta\leq s<1$ and $P_{c_*(\mu)}'(1)<0$, this solution is defined on $[0,\infty)$, with $q'>0$ and $q(z)\to1$ as $z\to\infty$. Together with $\eta(c_*(\mu),\mu)=0$, this gives \eqref{eq4.7}.

 By Lemma~\ref{le4.1}, $\eta$ is of class $C^2$ and
\[
\partial_c\eta(c,\mu)=\frac{d}{dc}P_c(\delta)-\frac{1}{\mu}<-\frac{1}{\mu}<0.
\]
 The implicit function theorem gives $c_*(\mu)\in C^2$ and
\[
c_*'(\mu)=-\frac{\partial_\mu\eta}{\partial_c\eta}
=\frac{c_*(\mu)/\mu^2}{\frac{1}{\mu}-\frac{d}{dc}P_{c_*(\mu)}(\delta)}
\in\left(0,\frac{c_*(\mu)}{\mu}\right).
\]
 Consequently,
\[
\left(\frac{c_*(\mu)}{\mu}\right)'
=\frac{c_*'(\mu)\mu-c_*(\mu)}{\mu^2}<0.
\]
\end{proof}

\begin{lemma}\label{le4.3}
	For every $\xi\in[\mu_0P_{c_*(\mu^*)}(\delta),\,\mu_0P_0(\delta))$, there exists a unique value $\mu=\mu(\xi)$ such that
	\[
	\frac{c_*(\mu(\xi))}{\mu(\xi)} = \frac{\xi}{\mu_0},
	\qquad \mu(c_*(\mu_0)) = \mu_0 = \frac{d}{\delta}.
	\]
	Then the function $g(\xi):= c_*(\mu(\xi))$ is of class $C^2$ and satisfies
	\begin{equation}
		g(c_*(\mu_0)) = c_*(\mu_0),\qquad
		g'(c_*(\mu_0)) =
		- \frac{d\,c_*(\mu_0)}
		{\displaystyle \mu_0^2 \int_0^\infty \bigl(q_{c_*(\mu_0)}'(z)\bigr)^2 e^{-\frac{c_*(\mu_0)}{d}z}\,dz}.
	\end{equation}
\end{lemma}

\begin{proof}
 Since $c_*(\mu)=\mu P_{c_*(\mu)}(\delta)\leq\mu P_0(\delta)$, we have $c_*(\mu)\to0$ as $\mu\to0$, and hence
	\[
	\lim_{\mu\to0}\frac{c_*(\mu)}{\mu}=P_0(\delta),\qquad
	\frac{c_*(\mu^*)}{\mu^*}=P_{c_*(\mu^*)}(\delta).
	\]
	Together with $\left(c_*(\mu)/\mu\right)'<0$, these identities and the inverse function theorem give the unique $C^2$ function $\mu=\mu(\xi)$ on the stated interval, with $\mu(c_*(\mu_0))=\mu_0$.
	Consequently, $g\in C^2$ with $g'(\xi)=c_*'(\mu(\xi))\mu'(\xi)<0$ and $g(c_*(\mu_0))=c_*(\mu_0)$.

 From the definitions of $g$ and $\mu(\xi)$, we have
	\[
	P_{g(\xi)}(\delta)=\frac{c_*(\mu(\xi))}{\mu(\xi)}=\frac{\xi}{\mu_0}.
	\]
	Differentiating this identity with respect to $\xi$ and taking $\xi=c_*(\mu_0)$, we obtain
	\begin{equation}\label{eq4.9}
		g'(c_*(\mu_0)) = c_*'(\mu_0)\mu'(c_*(\mu_0)) = \frac{1}{\mu_0\,\frac{d}{dc}P_{c_*(\mu_0)}(\delta)}.
	\end{equation}
 For brevity, write $c=c_*(\mu_0)$, $q=q_{c_*(\mu_0)}$ and $w(s)=\frac{d}{dc}P_c(s)$ in the rest of this proof. By \eqref{eq4.4},
	\[
	w'(s)=\frac{1}{d}+\frac{f(s)}{dP_c(s)^2}w(s),\qquad
	|w(s)|\leq\frac{1-s}{d}.
	\]
	Using $q'=P_c(q)$ and \eqref{eq4.7}, we obtain
	\[
	\frac{d}{dz}\left[e^{-\frac{c}{d}z}q'(z)w(q(z))\right]
	=\frac{1}{d}e^{-\frac{c}{d}z}\bigl(q'(z)\bigr)^2.
	\]
	Since $c>0$, $q'=P_c(q)$ is bounded and $|w(q)|\leq(1-q)/d$, there exists $C>0$ such that
\[
\left\{
\begin{aligned}
	&\int_0^\infty\frac{1}{d}e^{-\frac{c}{d}z}\bigl(q'(z)\bigr)^2\,dz
	\leq C\int_0^\infty e^{-\frac{c}{d}z}\,dz
	=\frac{Cd}{c}<\infty,\\
	&\left|e^{-\frac{c}{d}z}q'(z)w(q(z))\right|
	\leq Ce^{-\frac{c}{d}z}(1-q(z))
	\to0\quad\text{as }z\to\infty.
\end{aligned}
\right.
\]
	Integrating the differential identity above and using $q'(0)=c/\mu_0$, we find
	\[
	\frac{d}{dc}P_c(\delta)=w(\delta)
	=-\frac{1}{dq'(0)}\int_0^\infty\bigl(q'(z)\bigr)^2e^{-\frac{c}{d}z}\,dz
	=-\frac{\mu_0}{dc}\int_0^\infty\bigl(q'(z)\bigr)^2e^{-\frac{c}{d}z}\,dz.
	\]
	Substituting this into \eqref{eq4.9} yields the desired conclusion.

\end{proof}

Define
\begin{equation}\label{eq4.10}
	c_N(\delta) = d(N-1) \bigl[c_*(\mu_0)\bigl(1 - g'(c_*(\mu_0))\bigr)\bigr]^{-1}.
\end{equation}

\begin{proof}[Proof of Proposition \ref{pr1.5}]
 Write $c_\delta=c_*^\delta$, and denote by $q_\delta$ the corresponding semi-wave profile in Proposition A. By Proposition A, $0<c_\delta<c_0$ and $c_\delta\to c_0$ as $\delta\downarrow0$. Set
\[
I_\delta:=\int_0^\infty(q_\delta'(z))^2e^{-c_\delta z/d}\,dz.
\]
The formula for $g'$ and \eqref{eq4.10} give
\[
c_N(\delta)=\frac{d(N-1)}
{c_\delta\left(1+\dfrac{c_\delta\delta^2}{dI_\delta}\right)}.
\]
It therefore suffices to prove $I_\delta/\delta^2\to\infty$.

Since $q_\delta'>0$ and $f(q_\delta)>0$, we have $q_\delta''\leq(c_0/d)q_\delta'$. Using $q_\delta(0)=\delta$ and $q_\delta'(0)=c_\delta\delta/d$, we obtain
\[
\delta\leq q_\delta(z)\leq\delta e^{c_0z/d},\qquad
0<q_\delta'(z)\leq\frac{c_0\delta}{d}e^{c_0z/d}.
\]
For every $R>0$, these bounds give $q_\delta\to0$ uniformly on $[0,R]$, and hence $f(q_\delta)/q_\delta\to f'(0)$ uniformly. Dividing the equation for $q_\delta$ by $\delta$ and using continuous dependence for linear initial-value problems, we obtain $q_\delta/\delta\to v$ in $C^1([0,R])$, where
\[
dv''-c_0v'+f'(0)v=0,\qquad v(0)=1,\quad v'(0)=\frac{c_0}{d}.
\]
Recall that $c_0\geq2\sqrt{df'(0)}$. If $c_0^2=4df'(0)$, putting $\lambda=c_0/(2d)$ gives
\[
v(z)=(1+\lambda z)e^{\lambda z},\qquad
(v'(z))^2e^{-c_0z/d}=\lambda^2(2+\lambda z)^2.
\]
If $c_0^2>4df'(0)$, put $\lambda_\pm=(c_0\pm\sqrt{c_0^2-4df'(0)})/(2d)$. Then
\[
v(z)=\frac{\lambda_+e^{\lambda_+z}-\lambda_-e^{\lambda_-z}}
{\lambda_+-\lambda_-},\qquad
(v'(z))^2e^{-c_0z/d}
\sim\frac{\lambda_+^4}{(\lambda_+-\lambda_-)^2}e^{(\lambda_+-\lambda_-)z}.
\]
In both cases, $\int_0^\infty(v'(z))^2e^{-c_0z/d}\,dz=\infty$. For every $R>0$, the local convergence gives
\[
\liminf_{\delta\downarrow0}\frac{I_\delta}{\delta^2}
\geq\lim_{\delta\downarrow0}\int_0^R\left(\frac{q_\delta'(z)}{\delta}\right)^2e^{-c_\delta z/d}\,dz
=\int_0^R(v'(z))^2e^{-c_0z/d}\,dz.
\]
Letting $R\to\infty$ proves the required divergence and the stated limit.
\end{proof}

By the $C^2$ regularity of $g$, we obtain, as $t\to\infty$, the expansion
\begin{equation}\label{eq4.11}
	g(c_*(\mu_0) - c_N(\delta) t^{-1}) - g(c_*(\mu_0))
	= -g'(c_*(\mu_0))(c_N(\delta) t^{-1}) + O(c_N^2(\delta) t^{-2}).
\end{equation}
This identity will be crucial for the estimates in  Section~4.3.  For convenience, in the following, we denote $c_N:=c_N(\delta)$.

\subsection{Rough bounds }
	
\begin{lemma}\label{le4.4}
	For any \(c\in(0,c_*(\mu_0))\) and any \(\sigma\in(0,-f'(1))\), there exist positive constants \(T_*\), \(M_*\), and \(\sigma_*\in(0,-f'(1))\) such that for all \(t\ge T_*\),
	\[
	\begin{cases}
		 h(t) \ge c_*(\mu_0)t-M_*\log t, \\
		u(t,r) \le 1 + M_* e^{-\sigma t}, & r\in[0,h(t)], \\
		u(t,r) \ge 1 - M_* e^{-\sigma_* t}, & r\in[0,ct].
	\end{cases}
	\]
\end{lemma}
	
	\begin{proof}
{Write $c_*=c_*(\mu_0)$, and fix $c\in(0,c_*)$ and $\sigma\in(0,-f'(1))$.}

\textbf{Step 1.} {We first prove the upper bound and a lower bound for $u$.}
Let \(v\) be the solution to the ODE \(v'(t)=f(v)\) with initial data \(v(0)=\|u_0\|_{L^\infty}+1\).
{By the comparison principle, \(u(t,r)\le v(t)\) for \(0\le r\le h(t)\) and \(t\ge0\).}
By assumption \((\mathbf{f_m})\), there exists a constant \(M_1>0\) such that
\[
u(t,r)\le v(t)\le 1+M_1e^{-\sigma t}
\quad\text{for }0\le r\le h(t),\ t\ge0.
\]

\begingroup
For each $s\in(0,c_*)$, the construction in the proof of \cite[Lemma 3.3]{DN} gives a function $q^s$ satisfying
\[
\begin{cases}
d(q^s)''-s(q^s)'+f(q^s)=0,\quad (q^s)'>0,\quad z\in[0,z^s),\\
q^s(0)=\delta,\quad (q^s)'(0)=\dfrac{\delta}{d}c_*,\quad (q^s)'(z^s)=0,
\end{cases}
\]
where $z^s>0$ and $Q^s:=q^s(z^s)\in(\delta,1)$. Moreover, $Q^s\uparrow1$ as $s\uparrow c_*$.
Choose $s<s_1<s_2<c_*$ and $R>d(N-1)/(s_2-s_1)$, and define, for $t\ge0$,
\[
\left\{
\begin{aligned}
&L(t):=R+s_1t,\quad \underline h(t):=L(t)+z^{s_2},\\
&\underline u(t,r):=
\begin{cases}
q^{s_2}(\underline h(t)-r), & L(t)\le r\le\underline h(t),\\
Q^{s_2}, & 0\le r\le L(t).
\end{cases}
\end{aligned}
\right.
\]
For $L(t)<r<\underline h(t)$, a direct calculation gives
\[
\underline u_t-d\left(\underline u_{rr}+\frac{N-1}{r}\underline u_r\right)-f(\underline u)
=\left(s_1-s_2+\frac{d(N-1)}r\right)(q^{s_2})'(\underline h(t)-r)\le0.
\]
Moreover, $f(Q^{s_2})>0$ and
\[
\underline u_r(t,0)=0,\quad \underline u(t,\underline h(t))=\delta,\quad
\underline h'(t)=s_1<c_*=-\frac d\delta\underline u_r(t,\underline h(t)).
\]
By Lemma \ref{le3.2}, choose $T_1=T_1(s)\ge T_0$, with $T_0$ as in Lemma \ref{le3.1}, such that
\[
\underline h(0)<h(T_1),\quad
\underline u(0,r)\le Q^{s_2}<u(T_1,r)\quad\text{for }0\le r\le\underline h(0).
\]
The comparison principle yields
\[
h(t)\ge\underline h(t-T_1),\quad
u(t,r)\ge\underline u(t-T_1,r)
\quad\text{for }t\ge T_1,\ 0\le r\le\underline h(t-T_1).
\]
Since $s_1>s$, choose $T_2=T_2(s)\ge T_1(s)$ such that $L(t-T_1)\ge st$ for $t\ge T_2(s)$. Hence
\[
h(t)\ge st,\quad u(t,r)\ge Q^{s_2}>Q^s
\quad\text{for }t\ge T_2(s),\ 0\le r\le st.
\]
\par\endgroup

\textbf{Step 2.}
\begingroup
Choose $\hat c\in(0,\min\{c,c_*/\sqrt N\})$. We show that there exist $M_2>0$, $T_3>0$ and $\hat\sigma\in(0,\sigma)$ such that
\[
u(t,r)\ge1-M_2e^{-\hat\sigma t}
\quad\text{for }t\ge T_3,\ 0\le r\le\hat ct.
\]
Choose $\rho\in(0,1)$ such that
\begin{equation}\label{eq4.12}
\sigma\le-f'(w)\quad\text{for }1-\rho\le w\le1+\rho.
\end{equation}
Then
\begin{equation}\label{eq4.13}
f(w)\ge\sigma(1-w)\quad\text{for }1-\rho\le w\le1.
\end{equation}
Choose $\bar c\in(\sqrt N\hat c,c_*)$ close enough to $c_*$ that $Q^{\bar c}>1-\rho$. Take $\varepsilon\in(0,1)$ sufficiently small that, with $\theta:=\varepsilon^2\bar c^2/(4N)$,
\[
\sigma\theta<\frac1{2d},\quad
\frac{(1-\varepsilon)\bar c}{\sqrt N(1+\theta)}>\hat c.
\]
For $T\ge T_2(\bar c)$, let $\psi$ solve
\[
\begin{cases}
\psi_\tau-d\Delta\psi=-\sigma(\psi-1), & \tau>0,\ y\in D,\\
\psi=Q^{\bar c}, & \tau>0,\ y\in\partial D,\\
\psi(0,y)=Q^{\bar c}, & y\in D,
\end{cases}
\quad
D:=\left\{y\in\mathbb R^N:|y_i|<\frac{\bar cT}{\sqrt N},\ i=1,\ldots,N\right\}.
\]
By Step 1, $u(T+\tau,|y|)\ge Q^{\bar c}$ for $y\in\overline D$ and $\tau\ge0$. Since $Q^{\bar c}\le\psi\le1$, \eqref{eq4.13} and comparison give $u(T+\tau,|y|)\ge\psi(\tau,y)$. Repeating the argument in the proof of \cite[Lemma 3.2]{DHZ} with diffusion coefficient $d$, we obtain
\[
1-\psi(\theta T,y)\le C_0\left(e^{-\sigma\theta T}+e^{-T/(2d)}\right)
\le M_2e^{-\sigma\theta T}
\]
for $|y_i|\le(1-\varepsilon)\bar cT/\sqrt N$ and $T\ge T_2(\bar c)$, after increasing $T_2(\bar c)$ if necessary. Here $C_0$ and $M_2$ are independent of $T$. Set $\hat\sigma=\sigma\theta/(1+\theta)$ and choose $T_3\ge(1+\theta)T_2(\bar c)$. Taking $T=t/(1+\theta)$ proves the stated estimate for $t\ge T_3$.
\par\endgroup

\textbf{Step 3.}
\begingroup
We now prove the lower bounds in the lemma. Increase $T_3$ so that, for $t\ge T_3$,
\[
h(t)>\hat ct,\quad
M_2e^{-\hat\sigma t}\le\min\left\{\frac\rho2,\frac{1-\delta}2\right\}.
\]
For $t\ge T_3$, define $x(t)>0$ by
\[
(1-M_2e^{-\hat\sigma t})q_*(x(t))=\delta.
\]
Differentiation gives $x'(t)<0$. By the mean value theorem, there exists $C_1>0$ such that, for $t\ge T_3$,
\[
x(t)\le C_1M_2e^{-\hat\sigma t},\quad
\frac d\delta(1-M_2e^{-\hat\sigma t})q_*'(x(t))\ge c_*-C_1M_2e^{-\hat\sigma t}.
\]
Choose $\zeta_\rho>0$ such that $q_*(\zeta)\ge1-\rho/2$ for $\zeta\ge\zeta_\rho$, and then choose $\beta>0$ so that
\[
\beta\hat\sigma\ge C_1,\quad
\beta\hat\sigma\delta\min_{0\le\zeta\le\zeta_\rho}q_*'(\zeta)
\ge\max_{0\le w\le1}f'(w)+\hat\sigma.
\]
Choose $T_4\ge T_3$ such that
\[
\frac{d(N-1)}{\hat cT_4}+\beta\hat\sigma M_2e^{-\hat\sigma T_4}<c_*-\hat c.
\]
For $t\ge T_4$, define
\[
\left\{
\begin{aligned}
&\underline G(t):=\hat ct,\\
&\underline H(t):=c_*(t-T_4)+\hat cT_4
-\frac{d(N-1)}{\hat c}\log\frac{t}{T_4}
-\beta M_2\left(e^{-\hat\sigma T_4}-e^{-\hat\sigma t}\right),\\
&\underline U(t,r):=(1-M_2e^{-\hat\sigma t})q_*(\underline H(t)-r+x(t)),
\quad\underline G(t)\le r\le\underline H(t).
\end{aligned}
\right.
\]
By the choice of $T_4$, $\underline H'-\underline G'>0$ for $t\ge T_4$, and hence
\[
\underline H(T_4)=\underline G(T_4),\quad
\underline H(t)>\underline G(t)\quad\text{for }t>T_4.
\]
For the boundary conditions, Step 2 gives
\[
\underline U(t,\underline G(t))\le1-M_2e^{-\hat\sigma t}
\le u(t,\underline G(t)),\quad
\underline U(t,\underline H(t))=\delta.
\]
By the choice of $\beta$,
\[
\begin{aligned}
-\frac d\delta\underline U_r(t,\underline H(t))
&=\frac d\delta(1-M_2e^{-\hat\sigma t})q_*'(x(t))\\
&\ge c_*-C_1M_2e^{-\hat\sigma t}\ge\underline H'(t).
\end{aligned}
\]

We next verify the differential inequality. For this calculation, write $a=M_2e^{-\hat\sigma t}$, $\zeta=\underline H(t)-r+x(t)$ and $q=q_*(\zeta)$. Using $x'<0$, $r\ge\hat ct$, the equation for $q_*$ and the mean value theorem, we obtain
\[
\begin{aligned}
I(t,r)&:=\underline U_t-d\left(\underline U_{rr}+\frac{N-1}{r}\underline U_r\right)-f(\underline U)\\
&\le a\hat\sigma q-\beta\hat\sigma a(1-a)q_*'(\zeta)
+(1-a)f(q)-f((1-a)q)\\
&=a\left\{-f(q)+(f'(\eta)+\hat\sigma)q
-\beta\hat\sigma(1-a)q_*'(\zeta)\right\},
\end{aligned}
\]
where $\eta\in[(1-a)q,q]$. If $\zeta\ge\zeta_\rho$, then
\[
\eta\ge(1-a)q\ge\left(1-\frac\rho2\right)^2\ge1-\rho.
\]
Thus $f'(\eta)+\hat\sigma\le-\sigma+\hat\sigma<0$ by \eqref{eq4.12}. Since $f(q)\ge0$, we have $I(t,r)\le0$. For $0\le\zeta\le\zeta_\rho$, since $1-a\ge\delta$, then
\[
I(t,r)\le a\left\{\max_{0\le w\le1}f'(w)+\hat\sigma
-\beta\hat\sigma\delta\min_{0\le\zeta\le\zeta_\rho}q_*'(\zeta)\right\}\le0.
\]

We finally check the initial inequalities. At $t=T_4$, we have $\underline H(T_4)<h(T_4)$ and
\[
u(T_4,\underline G(T_4))\ge1-M_2e^{-\hat\sigma T_4}>\delta
=\underline U(T_4,\underline G(T_4)).
\]
The comparison principle yields
\[
u(t,r)\ge\underline U(t,r),\quad h(t)\ge\underline H(t)
\quad\text{for }t\ge T_4,\ \underline G(t)\le r\le\underline H(t).
\]
Since $x(t)>0$ and $q_*$ is increasing,
\[
u(t,r)\ge q_*(\underline H(t)-r)-M_2e^{-\hat\sigma t}
\quad\text{for }t\ge T_4,\ \underline G(t)\le r\le\underline H(t).
\]
The estimate $P_{c_*}(w)\ge m(1-w)$ in the proof of Lemma~\ref{le4.1} and $q_*'=P_{c_*}(q_*)$ give constants $M^*,\sigma^*>0$ such that
\[
1-q_*(z)\le M^*e^{-\sigma^*z}\quad\text{for }z\ge0.
\]
Fix $k\in(0,c_*-c)$ and take $0<\sigma_*<\min\{\hat\sigma,\sigma^*k\}$. Choose
\[
M_*>\max\left\{M_1,M_2+M^*,\frac{d(N-1)}{\hat c}\right\}.
\]
By the definition of $\underline H$, there exists $T_*\ge\max\{T_4,1\}$ such that
\[
\underline H(t)\ge c_*t-M_*\log t,\quad
\underline H(t)-ct\ge kt\quad\text{for }t\ge T_*.
\]
Consequently,
\[
u(t,r)\ge1-M^*e^{-\sigma^*kt}-M_2e^{-\hat\sigma t}
\ge1-M_*e^{-\sigma_*t}
\quad\text{for }t\ge T_*,\ \hat ct\le r\le ct.
\]
Together with Steps 1 and 2 and $h(t)\ge\underline H(t)$, this proves the lemma.
\par\endgroup
\end{proof}

\begin{remark}
	\begingroup
	By the proof of Lemma \ref{le4.4} and the monotonicity of $q_*$, we also obtain
	\begin{equation}\label{eq4.14}
	\begin{aligned}
	u(t,r)&\geq(1-M_2e^{-\hat{\sigma}t})q_*(c_*t-M_*\log t-r)\\
	&\geq(1-M_2e^{-\hat{\sigma}t})
	\left(1-M^*e^{-\sigma^*(c_*t-M_*\log t-r)}\right)
	\end{aligned}
	\end{equation}
	for $t\geq T_*$ and $r\in[\hat{c}t,c_*t-M_*\log t]$, where $M^*$ and $\sigma^*$ are as above.
	\par\endgroup
\end{remark}

	\subsection{Sharp bounds}
In this subsection, we construct upper and lower solutions using perturbed semi-wave solutions to obtain the bounds for $h(t)-c_*t+c_N\log t$.

{
Write $q(\mu,z)=q_{c_*(\mu)}(z)$ and $\xi(t)=c_*-c_Nt^{-1}$. Fix a closed interval $J\subset(0,\mu^*)$ containing $\mu_0$ in its interior, and take $t_0>1$ sufficiently large that $\mu(\xi(t))\in J$ for $t\geq t_0$.
}

\begin{lemma}\label{le4.6}
There exist \( C > 0 \) and \( T > 0 \) such that
\[
h(t) \geq c_* t - c_N \log t - C \quad \text{for } t \geq T,
\]
where \( c_N \) is given by Theorem \ref{th1.4}.
\end{lemma}

\begin{proof}
\begingroup
\textbf{Step 1.} We construct the lower solution and verify the free boundary conditions.
Fix $M>\max\{2M_*,4/\sigma^*\}$, where $M_*$ and $\sigma^*$ are given in \eqref{eq4.14}. For $t\geq t_0$, define
\[
\begin{cases}
\underline k(t)=c_*t-c_N\log t+Bt^{-1}\log t,\\[4pt]
\underline u(t,r)=q\bigl(\mu(\xi(t)),\underline k(t)-r\bigr)-t^{-2}\log t,
\end{cases}
\]
where $B>0$ will be chosen later. By \eqref{eq4.7}, the $C^1$ dependence of $q$ on $\mu$ and Lemma~\ref{le4.3}, there exists $C_0\geq1$ such that
\begin{equation}\label{eq4.15}
\begin{aligned}
&|q_\mu(\mu,z)|+|q_{zz}(\mu,z)|\leq C_0,\quad q_z(\mu,z)\geq C_0^{-1},
&&\mu\in J,\quad 0\leq z\leq1,\\
&|\mu'(\xi(t))|\leq C_0,&&t\geq t_0.
\end{aligned}
\end{equation}
Since $q(\mu,0)=\delta$, \eqref{eq4.15} gives, for all large $t$,
\[
\begin{aligned}
\underline u(t,\underline k(t))&=\delta-t^{-2}\log t<\delta,\\
\underline u(t,\underline k(t)-t^{-1})
&\geq\delta+C_0^{-1}t^{-1}-t^{-2}\log t>\delta.
\end{aligned}
\]
As $\underline u_r=-q_z<0$, there is a unique
$\underline h(t)\in(\underline k(t)-t^{-1},\underline k(t))$ satisfying
$\underline u(t,\underline h(t))=\delta$. The implicit function theorem gives $\underline h\in C^1$. The mean value theorem and \eqref{eq4.15} yield
\begin{equation}\label{eq4.16}
0<\underline k(t)-\underline h(t)\leq C_0t^{-2}\log t.
\end{equation}
Differentiating $\underline u(t,\underline h(t))=\delta$ and using \eqref{eq4.15}, we obtain
\begin{equation}\label{eq4.17}
\begin{aligned}
\underline h'(t)
&=\underline k'(t)+\frac{c_Nt^{-2}\mu'(\xi)q_\mu(\mu(\xi(t)),\underline k(t)-\underline h(t))+(2\log t-1)t^{-3}}{q_z(\mu(\xi(t)),\underline k(t)-\underline h(t))}\\
&=c_*-c_Nt^{-1}-Bt^{-2}\log t+O(t^{-2}).
\end{aligned}
\end{equation}
By $\mu_0q_z(\mu(\xi),0)=\xi$, \eqref{eq4.15} and \eqref{eq4.16}, we know
\[
\left|-\mu_0\underline u_r(t,\underline h(t))-\xi\right|
\leq\mu_0C_0|\underline k(t)-\underline h(t)|
\leq\mu_0C_0^2t^{-2}\log t.
\]
Thus, if $B>\mu_0C_0^2$, \eqref{eq4.17} implies, for all large $t$,
\begin{equation}\label{eq4.18}
\underline u(t,\underline h(t))=\delta,\quad
\underline h'(t)\leq-\mu_0\underline u_r(t,\underline h(t)).
\end{equation}

\textbf{Step 2.} We verify the differential inequality on
$\underline h(t)-M\log t\leq r\leq\underline h(t)$.
By \eqref{eq4.16} and the definition of $\underline k$, for $\underline h(t)-M\log t\leq r\leq\underline h(t)$,
\begin{equation}\label{eq4.19}
\begin{aligned}
r&\geq\underline h(t)-M\log t\\
&\geq c_*t-(c_N+M)\log t+(Bt-C_0)t^{-2}\log t\\
&\geq c_*t-(c_N+M)\log t
\end{aligned}
\end{equation}
for all large $t$ with $Bt\geq C_0$.
Differentiating $z=\int_\delta^{q(\mu,z)}P_{c_*(\mu)}(v)^{-1}\,dv$ with respect to $\mu$ gives
\begin{equation}\label{eq4.20}
q_\mu(\mu,z)=c_*'(\mu)P_{c_*(\mu)}(q(\mu,z))
\int_\delta^{q(\mu,z)}
\frac{\left.\partial_cP_c(v)\right|_{c=c_*(\mu)}}{P_{c_*(\mu)}(v)^2}\,dv.
\end{equation}
Choose $0<a<b<c_0$ with $c_*(J)\subset[a,b]$. Using \eqref{eq4.3}, the bound $|\partial_cP_c(v)|\leq(1-v)/d$ from \eqref{eq4.4}, and $\max_{\mu\in J}|c_*'(\mu)|<\infty$, we obtain, after increasing $C_0$,
\begin{equation}\label{eq4.21}
\begin{aligned}
|q_\mu(\mu,z)|
&\leq C_0(1-q(\mu,z))\int_\delta^{q(\mu,z)}\frac{dv}{1-v}\\
&=C_0(1-q(\mu,z))\log\frac{1-\delta}{1-q(\mu,z)}\leq C_0,
\quad\mu\in J,\quad z\geq0.
\end{aligned}
\end{equation}
For $\underline h(t)-M\log t\leq r\leq\underline h(t)$, the semi-wave equation \eqref{eq4.7} gives
\begin{equation}\label{eq4.22}
\begin{aligned}
&\underline u_t-d\underline u_{rr}-\frac{d(N-1)}r\underline u_r-f(\underline u)\\
&\quad=c_Nt^{-2}\mu'(\xi)q_\mu+(2\log t-1)t^{-3}
-q_zK+f(q)-f(q-t^{-2}\log t),
\end{aligned}
\end{equation}
where $q=q(\mu(\xi(t)),\underline k(t)-r)$, $q_z=q_z(\mu(\xi(t)),\underline k(t)-r)$, $q_\mu=q_\mu(\mu(\xi(t)),\underline k(t)-r)$, and
\[
K=g(\xi)-g(c_*)+c_Nt^{-1}+Bt^{-2}\log t-Bt^{-2}-\frac{d(N-1)}r.
\]
By \eqref{eq4.19},
\begin{equation}\label{eq4.23}
\frac{d(N-1)}r\leq\frac{d(N-1)}{c_*t}
+\frac{d(N-1)(c_N+M)}{c_*^2}t^{-2}\log t\,[1+o(1)].
\end{equation}
Substituting \eqref{eq4.11} and \eqref{eq4.23} into $K$, and using $c_N(1-g'(c_*))=d(N-1)/c_*$, we obtain
\begin{equation}\label{eq4.24}
K\geq\left[B-\frac{d(N-1)(c_N+M)}{c_*^2}+o(1)\right]t^{-2}\log t.
\end{equation}
Choose $\epsilon_0\in(0,(1-\delta)/2)$ and $\sigma_0>0$ such that
\begin{equation}\label{eq4.25}
f'(v)\leq-\sigma_0\quad\text{for }1-2\epsilon_0\leq v\leq1+2\epsilon_0,
\quad
\sigma_2:=\max_{\delta/2\leq v\leq1+\epsilon_0}|f'(v)|.
\end{equation}
If $q\geq1-\epsilon_0$, the mean value theorem gives
\[
f(q)-f(q-t^{-2}\log t)\leq-\sigma_0t^{-2}\log t.
\]
If $q\leq1-\epsilon_0$, then
\begin{equation}\label{eq4.26}
q_z=P_{c_*(\mu(\xi))}(q)
\geq\min_{\substack{\mu\in J\\\delta\leq v\leq1-\epsilon_0}}P_{c_*(\mu)}(v)
=:\sigma_1>0,
\end{equation}
and
\[
|f(q)-f(q-t^{-2}\log t)|\leq\sigma_2t^{-2}\log t.
\]
Fix
\[
B>\max\left\{\mu_0C_0^2,\,
\frac{d(N-1)(c_N+M)}{c_*^2}+\frac{\sigma_2}{\sigma_1}+1\right\}.
\]
Then \eqref{eq4.24} gives $K\geq(\sigma_2/\sigma_1+1)t^{-2}\log t$ for all large $t$.
Applying these two cases to \eqref{eq4.22}, and using \eqref{eq4.15} and \eqref{eq4.21}, we obtain
\begin{equation}\label{eq4.27}
\begin{aligned}
\underline u_t-d\underline u_{rr}-\frac{d(N-1)}r\underline u_r-f(\underline u)\leq-\min\{\sigma_0,\sigma_1\}t^{-2}\log t+O(t^{-2})\leq0
\end{aligned}
\end{equation}
for $\underline h(t)-M\log t\leq r\leq\underline h(t)$ and all large $t$.

\textbf{Step 3.} We verify the boundary conditions and the initial condition.
Since $c_*t-M_*\log t$ is increasing for all large $t$, the definition of $\underline k$ and \eqref{eq4.16} give, for every $s\geq0$,
\begin{equation}\label{eq4.28}
\begin{aligned}
&c_*(t+s)-M_*\log(t+s)-[\underline h(t)-M\log t]\\
&\quad\geq(c_N+M-M_*)\log t-Bt^{-1}\log t
\geq\frac M2\log t.
\end{aligned}
\end{equation}
By \eqref{eq4.16}, \eqref{eq4.28} and Lemma~\ref{le4.4},
\[
0<\underline h(t)-M\log t<c_*(t+s)-M_*\log(t+s)\leq h(t+s).
\]
If $\underline h(t)-M\log t\geq\hat c(t+s)$, then \eqref{eq4.14} and \eqref{eq4.28} imply
\[
\begin{aligned}
u(t+s,\underline h(t)-M\log t)
&\geq(1-M_2e^{-\hat\sigma(t+s)})(1-M^*t^{-\sigma^*M/2})\\
&\geq1-M_2e^{-\hat\sigma t}-M^*t^{-\sigma^*M/2}>1-t^{-2},
\end{aligned}
\]
where the last inequality follows from $\sigma^*M/2>2$.
If $\underline h(t)-M\log t<\hat c(t+s)$, Step 2 in the proof of Lemma~\ref{le4.4} gives
\[
u(t+s,\underline h(t)-M\log t)
\geq1-M_2e^{-\hat\sigma(t+s)}\geq1-M_2e^{-\hat\sigma t}>1-t^{-2}.
\]
Since $q<1$, it follows that, for all large $t$ and every $s\geq0$,
\begin{equation}\label{eq4.29}
\begin{aligned}
\underline u(t,\underline h(t)-M\log t)
&<1-t^{-2}\log t<1-t^{-2}\\
&<u(t+s,\underline h(t)-M\log t).
\end{aligned}
\end{equation}
Choose $T\geq t_0$ so that \eqref{eq4.18}, \eqref{eq4.27} and \eqref{eq4.29} hold for $t\geq T$, with $\underline h(t)-M\log t>0$. By Lemma~\ref{le3.1}, we may also take $T$ so that $u(t,r)\geq\delta$ for $t\geq T$ and $0\leq r\leq h(t)$.
By Lemma~\ref{le3.2}, $h(t)\to\infty$ and $u(t,r)\to1$ locally uniformly. Since $\underline u(T,r)<1$, there exists $T_1>0$ such that
\begin{equation}\label{eq4.30}
h(T+T_1)>\underline h(T),\quad
u(T+T_1,r)>\underline u(T,r)
\quad\text{for }\underline h(T)-M\log T\leq r\leq\underline h(T).
\end{equation}
Using \eqref{eq4.18}, \eqref{eq4.27} and \eqref{eq4.29} with $s=T_1$, together with \eqref{eq4.30}, Lemma~\ref{le2.2} yields, for $\underline h(t)-M\log t\leq r\leq\underline h(t)$ and $t\geq T$,
\begin{equation}\label{eq4.31}
\underline h(t)\leq h(t+T_1),\quad
\underline u(t,r)\leq u(t+T_1,r).
\end{equation}
Finally, \eqref{eq4.31} and \eqref{eq4.16} give, for $t\geq T+T_1$,
\[
\begin{aligned}
h(t)&\geq\underline h(t-T_1)=c_*(t-T_1)-c_N\log(t-T_1)+o(1)\\
&=c_*t-c_N\log t-c_*T_1+o(1).
\end{aligned}
\]
Thus $h(t)\geq c_*t-c_N\log t-C$ for all sufficiently large $t$, with $C=c_*T_1+1$.
\par\endgroup
\end{proof}

\begin{lemma}\label{le4.7}
There exist \( C > 0 \) and \( T > 0 \) such that
\[
h(t)\leq c_*t-c_N\log t+C\quad\text{for }t\geq T,
\]
where \( c_N \) is given by Theorem \ref{th1.4}.
\end{lemma}

\begin{proof}
\begingroup
\textbf{Step 1.} We construct the upper solution and verify the free boundary conditions.
For each $\mu\in J$, extend $q(\mu,\cdot)$ to the left of $z=0$ by solving
\[
\begin{cases}
dq_{zz}-c_*(\mu)q_z+f(q)=0,\\
q(\mu,0)=\delta,\quad q_z(\mu,0)=\dfrac{c_*(\mu)}{\mu}.
\end{cases}
\]
Since $J$ is compact, local existence and $C^1$ dependence on $\mu$ give a common interval $[-\ell,0]$, with $0<\ell\leq1$. As $\min_{\mu\in J}c_*(\mu)/\mu>0$, we may decrease $\ell$ and increase $C_0$ in \eqref{eq4.15} and \eqref{eq4.21} so that, for $\mu\in J$ and $-\ell\leq z\leq0$, we have
\begin{equation}\label{eq4.32}
q(\mu,z)\geq\delta/2,\quad
|q_\mu(\mu,z)|+|q_{zz}(\mu,z)|\leq C_0,\quad q_z(\mu,z)\geq C_0^{-1}.
\end{equation}
For $t\geq t_0$, define
\[
\begin{cases}
\bar k(t)=c_*t-c_N\log t-Bt^{-1}\log t,\\[4pt]
\bar u(t,r)=q\bigl(\mu(\xi(t)),\bar k(t)-r\bigr)+t^{-2}\log t,
\end{cases}
\]
where $B>0$ will be chosen later. For $t$ large enough that $t^{-1}\leq\ell$, \eqref{eq4.32} gives
\[
\begin{aligned}
\bar u(t,\bar k(t))&=\delta+t^{-2}\log t>\delta,\\
\bar u(t,\bar k(t)+t^{-1})&\leq\delta-C_0^{-1}t^{-1}+t^{-2}\log t<\delta.
\end{aligned}
\]
Since $\bar u_r=-q_z<0$, there is a unique $\bar h(t)\in(\bar k(t),\bar k(t)+t^{-1})$ satisfying $\bar u(t,\bar h(t))=\delta$. The implicit function theorem gives $\bar h\in C^1$, and the mean value theorem yields
\begin{equation}\label{eq4.33}
0<\bar h(t)-\bar k(t)\leq C_0t^{-2}\log t.
\end{equation}
Differentiating $\bar u(t,\bar h(t))=\delta$ and using \eqref{eq4.15} and \eqref{eq4.32}, we obtain
\begin{equation}\label{eq4.34}
\begin{aligned}
\bar h'(t)&=\bar k'(t)+\frac{c_Nt^{-2}\mu'(\xi)q_\mu(\mu(\xi(t)),\bar k(t)-\bar h(t))+(1-2\log t)t^{-3}}{q_z(\mu(\xi(t)),\bar k(t)-\bar h(t))}\\
&=c_*-c_Nt^{-1}+Bt^{-2}\log t+O(t^{-2}).
\end{aligned}
\end{equation}
By $\mu_0q_z(\mu(\xi),0)=\xi$, \eqref{eq4.32} and \eqref{eq4.33}, we have
\[
\left|-\mu_0\bar u_r(t,\bar h(t))-\xi\right|
\leq\mu_0C_0|\bar k(t)-\bar h(t)|\leq\mu_0C_0^2t^{-2}\log t.
\]
Thus, if $B>\mu_0C_0^2$, \eqref{eq4.34} and $\bar u_r<0$ give, for all large $t$,
\begin{equation}\label{eq4.35}
\begin{gathered}
\bar u(t,\bar h(t))=\delta,\quad \bar h'(t)\geq-\mu_0\bar u_r(t,\bar h(t)),\quad
\bar u(t,r)\geq\delta\quad\text{for }1\leq r\leq\bar h(t).
\end{gathered}
\end{equation}

\textbf{Step 2.} We verify the differential inequality on $1\leq r\leq\bar h(t)$.
The semi-wave equation \eqref{eq4.7} gives
\begin{equation}\label{eq4.36}
\begin{aligned}
&\bar u_t-d\bar u_{rr}-\frac{d(N-1)}r\bar u_r-f(\bar u)\\
&\quad=c_Nt^{-2}\mu'(\xi)q_\mu+(1-2\log t)t^{-3}
-q_z\hat K+f(q)-f(q+t^{-2}\log t),
\end{aligned}
\end{equation}
where $q=q(\mu(\xi(t)),\bar k(t)-r)$, $q_z=q_z(\mu(\xi(t)),\bar k(t)-r)$, $q_\mu=q_\mu(\mu(\xi(t)),\bar k(t)-r)$, and
\[
\hat K=g(\xi)-g(c_*)+c_Nt^{-1}-Bt^{-2}\log t+Bt^{-2}-\frac{d(N-1)}r.
\]
For $1\leq r\leq\bar h(t)$, \eqref{eq4.33} and the definition of $\bar k$ give
\begin{equation}\label{eq4.37}
\frac{d(N-1)}r\geq\frac{d(N-1)}{\bar h(t)}
=\frac{d(N-1)}{c_*t}+\frac{d(N-1)c_N}{c_*^2}t^{-2}\log t\,[1+o(1)].
\end{equation}
Substituting \eqref{eq4.11} and \eqref{eq4.37} into $\hat K$, and using $c_N(1-g'(c_*))=d(N-1)/c_*$, we obtain
\begin{equation}\label{eq4.38}
\hat K\leq-\left[B+\frac{d(N-1)c_N}{c_*^2}+o(1)\right]t^{-2}\log t.
\end{equation}
Take $\epsilon_0,\sigma_0,\sigma_2$ from \eqref{eq4.25} and $\sigma_1$ from \eqref{eq4.26}, decreasing $\sigma_1$ if necessary so that $\sigma_1\leq C_0^{-1}$.
By \eqref{eq4.25}, \eqref{eq4.26}, \eqref{eq4.32} and the mean value theorem,
\[
\begin{cases}
f(q)-f(q+t^{-2}\log t)\geq\sigma_0t^{-2}\log t,
&q\geq1-\epsilon_0,\\[4pt]
q_z\geq\sigma_1,\quad
f(q)-f(q+t^{-2}\log t)\geq-\sigma_2t^{-2}\log t,
&q\leq1-\epsilon_0.
\end{cases}
\]
Fix $B>\max\{\mu_0C_0^2,\sigma_2/\sigma_1+1\}$. Then \eqref{eq4.38} gives
$-\hat K\geq(\sigma_2/\sigma_1+1)t^{-2}\log t$ for all large $t$.
Applying these two cases to \eqref{eq4.36}, and using \eqref{eq4.15}, \eqref{eq4.21} and \eqref{eq4.32}, we obtain
\begin{equation}\label{eq4.39}
\begin{aligned}
\bar u_t-d\bar u_{rr}-\frac{d(N-1)}r\bar u_r-f(\bar u)\geq\min\{\sigma_0,\sigma_1\}t^{-2}\log t+O(t^{-2})\geq0
\end{aligned}
\end{equation}
for $1\leq r\leq\bar h(t)$ and all large $t$.

\textbf{Step 3.}  We verify the boundary conditions and the initial condition.
By \eqref{eq4.3} and $q_z(\mu,z)=P_{c_*(\mu)}(q(\mu,z))$, there exist $C_1,\eta>0$ such that
\begin{equation}\label{eq4.40}
1-q(\mu,z)\leq C_1e^{-\eta z},\quad\mu\in J,\quad z\geq0.
\end{equation}
Since $\bar k(t)/t\to c_*>0$, applying \eqref{eq4.40} at $z=\bar k(t)-1$ gives
\begin{equation}\label{eq4.41}
\bar u(t,1)\geq1-C_1e^{-\eta(\bar k(t)-1)}+t^{-2}\log t\geq1+t^{-2}
\end{equation}
for all large $t$. Choose $S_0\geq t_0$ so that \eqref{eq4.35}, \eqref{eq4.39} and \eqref{eq4.41} hold for $t\geq S_0$, with $\bar h(t)>1$.
By Lemma~\ref{le2.5}, $h(t)\leq h_0+C_2t$. Fix $L>1$ such that $c_*L>C_2+2$. For $T\geq S_0$, set $\tau=(L-1)T$, so that $T+\tau=LT$.
At $t=T$,
\[
\bar k(LT)-h(T)
\geq(c_*L-C_2)T-h_0-c_N\log(LT)-B(LT)^{-1}\log(LT)\geq T
\]
for all sufficiently large $T$. Hence $\bar h(T+\tau)>h(T)$, and \eqref{eq4.40} and Lemma~\ref{le4.4} give
\begin{equation}\label{eq4.42}
\bar u(T+\tau,r)\geq1-C_1e^{-\eta T}+(LT)^{-2}\log(LT)
>1+M_*e^{-\sigma T}\geq u(T,r)
\end{equation}
for $1\leq r\leq h(T)$.
Moreover, $t\leq t+\tau\leq Lt$ for $t\geq T$. By \eqref{eq4.41} and Lemma~\ref{le4.4},
\begin{equation}\label{eq4.43}
\bar u(t+\tau,1)\geq1+(t+\tau)^{-2}\geq1+L^{-2}t^{-2}
\geq1+M_*e^{-\sigma t}\geq u(t,1)
\end{equation}
for all $t\geq T$, provided that $T$ is large enough. Increasing $T$ if necessary, we may assume that $h(t)>1$ for $t\geq T$.
By \eqref{eq4.35}, \eqref{eq4.39}, \eqref{eq4.42} and \eqref{eq4.43}, Lemma~\ref{le2.2} yields
\[
h(t)\leq\bar h(t+\tau),\quad u(t,r)\leq\bar u(t+\tau,r)
\quad\text{for }t\geq T,\quad1\leq r\leq h(t).
\]
Finally, by \eqref{eq4.33} and the definition of $\bar k$,
\[
\begin{aligned}
h(t)&\leq\bar h(t+\tau)=c_*(t+\tau)-c_N\log(t+\tau)+o(1)\\
&=c_*t-c_N\log t+c_*\tau+o(1).
\end{aligned}
\]
This proves the upper bound.
\par\endgroup
\end{proof}

	\subsection{Convergence}
In this subsection, we complete the proof of Theorem \ref{th1.4}.

\begingroup
By Lemmas \ref{le4.6} and \ref{le4.7}, there exist $C,T>0$ such that
\[
-C\leq h(t)-c_*t+c_N\log t\leq C\quad\text{for }t\geq T.
\]
For $t\geq T$ and $-k(t)\leq r\leq l(t)$, set
\[
k(t)=c_*t-c_N\log t-2C,\quad
l(t)=h(t)-k(t),\quad
\phi(t,r)=u(t,r+k(t)).
\]
Then
\begin{equation}\label{eq4.44}
C\leq l(t)\leq3C\quad\text{for }t\geq T.
\end{equation}

Increasing $T$ if necessary, we may assume that $k(t)>0$ for $t\geq T$. The pair $(\phi,l)$ satisfies
\begin{equation}\label{eq4.45}
\begin{cases}
\displaystyle
\phi_t-d\phi_{rr}
-\left[c_*-c_Nt^{-1}+\frac{d(N-1)}{r+k(t)}\right]\phi_r=f(\phi),
&-k(t)<r<l(t),\ t>T,\\[5pt]
\phi_r(t,-k(t))=0,\quad \phi(t,l(t))=\delta,
&t>T,\\[3pt]
l'(t)=-\mu_0\phi_r(t,l(t))-c_*+c_Nt^{-1},
&t>T.
\end{cases}
\end{equation}
Choose $t_n\to\infty$ so that
\begin{equation}\label{eq4.46}
l(t_n)\longrightarrow\liminf_{t\to\infty}l(t),
\end{equation}
and define
\[
k_n(t)=k(t+t_n),\quad
l_n(t)=l(t+t_n),\quad
\phi_n(t,r)=\phi(t+t_n,r),\quad t>T-t_n.
\]

\begin{lemma}\label{le4.8}
After passing to a subsequence, for some $\alpha\in(0,1)$,
\[
l_n\to L\quad\text{in }C_{\mathrm{loc}}^{1+\alpha/2}(\mathbb R),
\quad
\phi_n\to\Phi\quad\text{in }C_{\mathrm{loc}}^{(1+\alpha)/2,\,1+\alpha}(\Omega),
\]
where $\Omega=\{(t,r):t\in\mathbb R,\ r<L(t)\}$. The pair $(\Phi,L)$ satisfies
\begin{equation}\label{eq4.47}
\begin{cases}
\Phi_t-d\Phi_{rr}-c_*\Phi_r=f(\Phi),\quad \delta<\Phi\leq1,
&t\in\mathbb R,\ r<L(t),\\[3pt]
\Phi(t,L(t))=\delta,\quad
L'(t)=-\mu_0\Phi_r(t,L(t))-c_*,
&t\in\mathbb R.
\end{cases}
\end{equation}
Moreover, $L(t)\equiv L_0\in[C,3C]$ and $\Phi(t,r)=q_*(L_0-r)$. In addition,
\[
\lim_{n\to\infty}\sup_{0\leq r\leq h(t_n)}
|u(t_n,r)-q_*(h(t_n)-r)|=0,
\quad
h'(t_n+\cdot)\to c_*
\quad\text{in }C_{\mathrm{loc}}^{\alpha/2}(\mathbb R).
\]
\end{lemma}

\begin{proof}
\textbf{Step 1.} We obtain the limiting problem.
By Lemma~\ref{le2.5},
\[
|l_n'(t)|\leq C_2+c_*+c_N/T\quad\text{for }t>T-t_n.
\]
For $t>T-t_n$ and $-k_n(t)/l_n(t)\leq s\leq1$, set
\[
w_n(t,s)=\phi_n(t,l_n(t)s).
\]
Then \eqref{eq4.45} gives
\begin{equation}\label{eq4.48}
\begin{cases}
\displaystyle
(w_n)_t-\frac{d}{l_n(t)^2}(w_n)_{ss}
-\frac{\widehat C_n(t,s)}{l_n(t)}(w_n)_s=f(w_n),
&t>T-t_n,\quad-\dfrac{k_n(t)}{l_n(t)}<s<1,\\[6pt]
w_n(t,1)=\delta,\quad
l_n'(t)=-\mu_0\dfrac{(w_n)_s(t,1)}{l_n(t)}-c_*+\dfrac{c_N}{t+t_n},
&t>T-t_n,
\end{cases}
\end{equation}
where
\[
\widehat C_n(t,s)
=sl_n'(t)+c_*-\frac{c_N}{t+t_n}
+\frac{d(N-1)}{l_n(t)s+k_n(t)}.
\]
For each $M>0$ and $T_0\in\mathbb R$, take $n$ sufficiently large that, for $(t,s)\in[T_0-1,T_0+1]\times[-M-1,1]$,
\[
l_n(t)s+k_n(t)\geq k_n(t)-3C(M+1)\geq1.
\]
By \eqref{eq4.44} and Lemma~\ref{le2.4}, for $(t,s)\in[T_0-1,T_0+1]\times[-M-1,1]$ we have
\begin{equation}\label{eq4.49}
\frac{d}{9C^2}\leq\frac{d}{l_n(t)^2}\leq\frac{d}{C^2},\quad
|\widehat C_n(t,s)|+|w_n(t,s)|\leq C_M.
\end{equation}
By \eqref{eq4.49}, the interior and boundary $L^p$ estimates for \eqref{eq4.48} give
\[
\|w_n\|_{W_p^{1,2}([T_0,T_0+1]\times[-M,1])}\leq C_{M,p},
\]
where $C_M$ and $C_{M,p}$ are independent of $n$ and $T_0$. By \eqref{eq4.48}, we obtain, for each fixed $\gamma\in(0,1)$,
\begin{equation}\label{eq4.50}
\|w_n\|_{C^{(1+\gamma)/2,\,1+\gamma}([T_0,T_0+1]\times[-M,1])}
+\|l_n\|_{C^{1+\gamma/2}([T_0,T_0+1])}\leq C_M.
\end{equation}
Thus, for $0<\alpha<\gamma$, a diagonal subsequence satisfies
\begin{equation}\label{eq4.51}
w_n\to W\quad\text{in }C_{\mathrm{loc}}^{(1+\alpha)/2,\,1+\alpha}
(\mathbb R\times(-\infty,1]),
\quad
l_n\to L\quad\text{in }C_{\mathrm{loc}}^{1+\alpha/2}(\mathbb R).
\end{equation}
Since $k_n(t)\to\infty$ locally uniformly in $t$,
\[
\frac{d}{l_n(t)^2}\longrightarrow\frac{d}{L(t)^2},\quad
\frac{\widehat C_n(t,s)}{l_n(t)}
\longrightarrow\frac{sL'(t)+c_*}{L(t)}
\]
locally uniformly for $(t,s)\in\mathbb R\times(-\infty,1]$.
By standard parabolic regularity theory, the pair $(\Phi,L)$, where $\Phi(t,r)=W(t,r/L(t))$, satisfies \eqref{eq4.47}.
By Lemmas~\ref{le3.1} and \ref{le4.4}, $\delta\leq\Phi\leq1$. Since $f(\delta)>0$, the strong maximum principle gives $\Phi>\delta$ for $r<L(t)$.
Moreover, \eqref{eq4.44} and \eqref{eq4.46} give
\begin{equation}\label{eq4.52}
C\leq L(t)\leq3C,\quad
L(t)\geq L(0)=\liminf_{s\to\infty}l(s)
\quad\text{for }t\in\mathbb R.
\end{equation}

\textbf{Step 2.} We identify $(\Phi,L)$.
By \eqref{eq4.31}, there exist constants $T_1>0$ and $M>2M_*$ such that, for all large $t$,
\begin{equation}\label{eq4.53}
\begin{aligned}
\underline h(t)\leq h(t+T_1),\quad
u(t+T_1,r)\geq\underline u(t,r)
\quad\text{for }\underline h(t)-M\log t\leq r\leq\underline h(t).
\end{aligned}
\end{equation}
By the definitions of $k,\underline k$ and \eqref{eq4.16}, we have
\[
\left\{
\begin{aligned}
	&\underline h(t+t_n-T_1)-k_n(t)\longrightarrow2C-c_*T_1,\\
	&\underline h(t+t_n-T_1)-M\log(t+t_n-T_1)-k_n(t)\longrightarrow-\infty.
\end{aligned}
\right.
\]
locally uniformly in $t\in\mathbb R$ as $n\to\infty$.
Choose $R^0<\min\{C,2C-c_*T_1\}$. By the definition of $\underline u$,
\[
\underline u(t+t_n-T_1,r+k_n(t))
\longrightarrow q_*(2C-c_*T_1-r)
\]
locally uniformly for $(t,r)\in\mathbb R\times(-\infty,R^0]$.
Since $\phi_n(t,r)=u(t+t_n,r+k_n(t))$, \eqref{eq4.53} and the monotonicity of $q_*$ yield
\begin{equation}\label{eq4.54}
\Phi(t,r)\geq q_*(R^0-r)
\quad\text{for }t\in\mathbb R,\ r\leq R^0.
\end{equation}
Since $\Phi\leq1$, \eqref{eq4.54} gives
\begin{equation}\label{eq4.55}
\sup_{t\in\mathbb R}|1-\Phi(t,r)|\longrightarrow0
\quad\text{as }r\to-\infty.
\end{equation}
Define
\[
R^*=\sup\left\{R\leq L(0):
\Phi(t,r)\geq q_*(R-r)
\text{ for all }t\in\mathbb R,\ r\leq R\right\}.
\]
This set is nonempty by \eqref{eq4.54}, and continuity gives
\[
\Phi(t,r)\geq q_*(R^*-r)
\quad\text{for }t\in\mathbb R,\ r\leq R^*.
\]
Passing to the limit in \eqref{eq4.50}, for every $R>0$ and $T_0\in\mathbb R$ we have
\[
\|W\|_{C^{(1+\gamma)/2,\,1+\gamma}([T_0,T_0+1]\times[-R,1])}
+\|L\|_{C^{1+\gamma/2}([T_0,T_0+1])}\leq C_R.
\]
By \eqref{eq4.47}, \eqref{eq4.52}, \eqref{eq4.54} and \eqref{eq4.55}, the sliding argument in \cite[Lemma 3.7]{DN} gives $R^*=L(0)$.
Set $L_0=L(0)$ and $w(t,r)=\Phi(t,r)-q_*(L_0-r)$ for $t\in\mathbb R$ and $r\leq L_0$. Then $w\geq0$ and $w(0,L_0)=0$. Since $L'(0)=0$, \eqref{eq4.47} gives
\[
w_r(0,L_0)=\Phi_r(0,L_0)+q_*'(0)=0.
\]
The strong maximum principle and the Hopf boundary lemma yield
\[
\Phi(t,r)=q_*(L_0-r)\quad\text{for }t\leq0,\ r\leq L_0.
\]
Since $\Phi>\delta$ for $r<L(t)$, we also have $L(t)=L_0$ for $t\leq0$. Both $(\Phi,L)$ and $(q_*(L_0-r),L_0)$ solve \eqref{eq4.47} with the same initial data at $t=0$. By uniqueness,
\[
L(t)=L_0,\quad \Phi(t,r)=q_*(L_0-r)
\quad\text{for }t\geq0,\ r\leq L_0.
\]

\textbf{Step 3.} We complete the proof.
We first show that
\begin{equation}\label{eq4.56}
\lim_{A\to\infty}\limsup_{t\to\infty}
\sup_{0\leq r\leq h(t)-A}|u(t,r)-1|=0.
\end{equation}
By \eqref{eq4.44}, \eqref{eq4.53} and \eqref{eq4.16}, for all large $t$,
\begin{equation}\label{eq4.57}
0\leq h(t)-\underline h(t-T_1)\leq C+c_*T_1+o(1)
\leq D_0:=C+c_*T_1+1.
\end{equation}
Since $M>2M_*$,
\[
\begin{aligned}
&c_*t-M_*\log t-\underline h(t-T_1)+M\log(t-T_1)\\
&\quad=(c_N+M-M_*)\log t+c_*T_1+o(1)\geq\frac M2\log t.
\end{aligned}
\]
Applying Lemma~\ref{le4.4} for $0\leq r\leq\hat ct$ and \eqref{eq4.14} for
$\hat ct\leq r\leq\underline h(t-T_1)-M\log(t-T_1)$, we obtain
\begin{equation}\label{eq4.58}
\begin{aligned}
\inf_{0\leq r\leq\underline h(t-T_1)-M\log(t-T_1)}u(t,r)\geq1-M_*e^{-\sigma_*t}-M_2e^{-\hat\sigma t}
-M^*t^{-\sigma^*M/2}\longrightarrow1.
\end{aligned}
\end{equation}
Fix $A>D_0$. For all large $t$ and
\[
\underline h(t-T_1)-M\log(t-T_1)\leq r\leq h(t)-A,
\]
\eqref{eq4.57} gives
\[
r\leq\underline h(t-T_1),\quad
\underline k(t-T_1)-r\geq\underline h(t-T_1)-r\geq A-D_0.
\]
By \eqref{eq4.53} and $q_z>0$,
\[
u(t,r)\geq q\bigl(\mu(\xi(t-T_1)),A-D_0\bigr)
-(t-T_1)^{-2}\log(t-T_1).
\]
Using $\mu(\xi(t-T_1))\to\mu_0$, \eqref{eq4.58} and the upper bound in Lemma~\ref{le4.4}, we obtain, for each $A>D_0$,
\[
\limsup_{t\to\infty}\sup_{0\leq r\leq h(t)-A}|u(t,r)-1|
\leq1-q_*(A-D_0).
\]
Letting $A\to\infty$ proves \eqref{eq4.56}.

By \eqref{eq4.51} and $\Phi(t,r)=q_*(L_0-r)$,
\begin{equation}\label{eq4.59}
u(t_n,h(t_n)-z)=w_n\left(0,1-\frac{z}{l_n(0)}\right)
\longrightarrow W\left(0,1-\frac z{L_0}\right)=q_*(z)
\end{equation}
locally uniformly for $z\geq0$. For each fixed $A>0$ and $0\leq r\leq h(t_n)-A$,
\[
|u(t_n,r)-q_*(h(t_n)-r)|
\leq |u(t_n,r)-1|+1-q_*(A).
\]
Together with \eqref{eq4.59}, this gives
\[
\begin{aligned}
\limsup_{n\to\infty}\sup_{0\leq r\leq h(t_n)}
|u(t_n,r)-q_*(h(t_n)-r)|\leq\limsup_{t\to\infty}\sup_{0\leq r\leq h(t)-A}|u(t,r)-1|+1-q_*(A).
\end{aligned}
\]
Letting $A\to\infty$ and using \eqref{eq4.56} proves the uniform convergence.
Finally, by \eqref{eq4.51} and $L(t)\equiv L_0$,
\[
h'(t_n+t)=l_n'(t)+c_*-\frac{c_N}{t_n+t}\longrightarrow c_*
\quad\text{in }C_{\mathrm{loc}}^{\alpha/2}(\mathbb R).
\]
\end{proof}

\begin{lemma}\label{le4.9}
There exists $\hat h\in\mathbb R$ such that
\[
\lim_{t\to\infty}[h(t)-c_*t+c_N\log t]=\hat h=L_0-2C,
\]
where $L_0$ is given by Lemma \ref{le4.8}.
\end{lemma}

\begin{proof}
By Lemma~\ref{le4.8}, $l(t_n)\to L_0=\liminf_{t\to\infty}l(t)$.
Set $\hat h=L_0-2C$. It suffices to prove
\begin{equation}\label{eq4.60}
\limsup_{t\to\infty}l(t)\leq L_0.
\end{equation}

\textbf{Step 1.} We construct the upper solution.
Recall that $\xi(t)=c_*-c_Nt^{-1}$ and $\mu(\xi(t))\in J$ for $t\geq t_0$.
Write
\[
Q(\zeta,z)=q(\mu(\zeta),z),\quad
\xi(t_0)\leq\zeta\leq c_*,\quad z\geq0.
\]
By \eqref{eq4.7} and Lemma~\ref{le4.3},
\begin{equation}\label{eq4.61}
\begin{cases}
dQ_{zz}-g(\zeta)Q_z+f(Q)=0,\quad z>0,\\[4pt]
\delta\leq Q(\zeta,z)<1,\quad Q_z(\zeta,z)>0,\quad z\geq0,\\[4pt]
Q(\zeta,0)=\delta,\quad Q_z(\zeta,0)=\dfrac{\zeta}{\mu_0}.
\end{cases}
\end{equation}
By \eqref{eq4.4}, \eqref{eq4.20}, \eqref{eq4.21}
and $\mu'<0$, there exists $C_Q>0$ such that, for $\xi(t_0)\leq\zeta\leq c_*$ and $z\geq0$,
\begin{equation}\label{eq4.62}
\begin{aligned}
0\leq Q_\zeta(\zeta,z)=q_\mu(\mu(\zeta),z)\mu'(\zeta)\leq C_Q,\quad |Q(\zeta,z)-q_*(z)|\leq C_Q|\zeta-c_*|.
\end{aligned}
\end{equation}

Choose $\rho\in(0,(1-\delta)/2)$ and $\beta>0$ such that
\begin{equation}\label{eq4.63}
f'(v)\leq-2\beta\quad\text{for }v\in[1-\rho,1+\rho].
\end{equation}
Let
\begin{equation}\label{eq4.64}
K_f=\max_{v\in[\delta,1+\rho]}|f'(v)|,\quad
m=\min_{\substack{\xi(t_0)\leq\zeta\leq c_*\\\delta\leq v\leq1-\rho}}
P_{g(\zeta)}(v)>0.
\end{equation}
Fix $B>0$ such that
\begin{equation}\label{eq4.65}
B\beta>c_*,\quad B\beta m>(K_f+\beta)(1-\delta),
\end{equation}
and fix $0<\varepsilon_0\leq\min\{1,\rho/(1-\delta)\}$.
For $0<\varepsilon\leq\varepsilon_0$ and $t\geq0$, define
\[
\begin{cases}
\bar h_n(t)=k(t+t_n)+L_0+\varepsilon+B\varepsilon(1-e^{-\beta t}),\\[5pt]
\bar u_n(t,r)=\delta+(1+\varepsilon e^{-\beta t})
[Q(\xi(t+t_n),\bar h_n(t)-r)-\delta]
\end{cases}
\]
for $0\leq r\leq\bar h_n(t)$.
Choose $n_0$ sufficiently large that $t_n\geq t_0$ and $\bar h_n(t)>0$
for $n\geq n_0$, $t\geq0$ and $0<\varepsilon\leq\varepsilon_0$.
By \eqref{eq4.61},
\[
\bar u_n(t,\bar h_n(t))=\delta,\quad
\bar u_n(t,r)>\delta\quad\text{for }0\leq r<\bar h_n(t).
\]
By \eqref{eq4.65},
\begin{equation}\label{eq4.66}
\bar h_n'(t)+\mu_0(\bar u_n)_r(t,\bar h_n(t))
=\varepsilon e^{-\beta t}[B\beta-\xi(t+t_n)]>0.
\end{equation}

\textbf{Step 2.} We verify the differential inequality.
For $t\geq0$ and $0<r\leq\bar h_n(t)$, set
\[
J_n(t,r)=\xi(t+t_n)-g(\xi(t+t_n))+\frac{d(N-1)}r.
\]
The definition of $\bar h_n$ gives
\[
|\bar h_n(t)-c_*(t+t_n)+c_N\log(t+t_n)|
\leq|\hat h|+(B+1)\varepsilon_0.
\]
By \eqref{eq4.11} and $c_N(1-g'(c_*))=d(N-1)/c_*$, there exists $C_3>0$ such that,
for all sufficiently large $n$, $t\geq0$ and $0<\varepsilon\leq\varepsilon_0$,
\[
\left\{
\begin{aligned}
&\left|\xi(t+t_n)-g(\xi(t+t_n))+\frac{d(N-1)}{c_*(t+t_n)}\right|
\leq\frac{C_3}{(t+t_n)^2},\\[5pt]
&\left|\frac{d(N-1)}{\bar h_n(t)}-\frac{d(N-1)}{c_*(t+t_n)}
-\frac{d(N-1)c_N\log(t+t_n)}{c_*^2(t+t_n)^2}\right|
\leq\frac{C_3}{(t+t_n)^2}.
\end{aligned}
\right.
\]
Increasing $n_0$ if necessary, for $n\geq n_0$, $t\geq0$ and $0<r\leq\bar h_n(t)$ we obtain
\begin{equation}\label{eq4.67}
\begin{aligned}
J_n(t,r)&\geq\xi(t+t_n)-g(\xi(t+t_n))+\frac{d(N-1)}{\bar h_n(t)}\\
&\geq\frac{1}{(t+t_n)^2}
\left[\frac{d(N-1)c_N}{c_*^2}\log(t+t_n)-2C_3\right]>0.
\end{aligned}
\end{equation}
For the following calculation, write $a(t)=\varepsilon e^{-\beta t}$ and
\[
\begin{aligned}
Q&=Q(\xi(t+t_n),\bar h_n(t)-r),\\
Q_z&=Q_z(\xi(t+t_n),\bar h_n(t)-r),\quad
Q_\zeta=Q_\zeta(\xi(t+t_n),\bar h_n(t)-r).
\end{aligned}
\]
By \eqref{eq4.61}, \eqref{eq4.62},
\eqref{eq4.67} and $f(Q)\geq0$, for $t\geq0$ and
$0<r<\bar h_n(t)$, we have
\begin{equation}
\begin{aligned}
&(\bar u_n)_t-d(\bar u_n)_{rr}-\frac{d(N-1)}r(\bar u_n)_r-f(\bar u_n)\\
&\quad=\frac{c_N(1+a)}{(t+t_n)^2}Q_\zeta
+(1+a)Q_z(J_n+B\beta a)\\
&\quad\quad +(1+a)f(Q)-f(Q+a(Q-\delta))-\beta a(Q-\delta)\\
&\quad\geq B\beta aQ_z+f(Q)-f(Q+a(Q-\delta))-\beta a(Q-\delta).
\end{aligned}
\end{equation}
Since $0\leq a(Q-\delta)\leq\varepsilon_0(1-\delta)\leq\rho$, if $Q\geq1-\rho$, then
\[
1-\rho\leq Q\leq Q+a(Q-\delta)\leq1+\rho.
\]
By \eqref{eq4.63} and the mean value theorem,
\[
f(Q)-f(Q+a(Q-\delta))\geq2\beta a(Q-\delta).
\]
If $Q\leq1-\rho$, then \eqref{eq4.64} gives
\[
\left\{
\begin{aligned}
&Q_z=P_{g(\xi(t+t_n))}(Q)\geq m,\\[3pt]
&f(Q)-f(Q+a(Q-\delta))\geq-K_fa(Q-\delta).
\end{aligned}
\right.
\]
Thus, using $0<Q-\delta\leq1-\delta$ and \eqref{eq4.65}, we obtain
\begin{equation}\label{eq4.69}
\begin{aligned}
&(\bar u_n)_t-d(\bar u_n)_{rr}-\frac{d(N-1)}r(\bar u_n)_r-f(\bar u_n)\\
&\quad\geq
\begin{cases}
\beta a(Q-\delta)>0,&Q\geq1-\rho,\\[3pt]
a[B\beta m-(K_f+\beta)(1-\delta)]>0,&Q\leq1-\rho.
\end{cases}
\end{aligned}
\end{equation}

\textbf{Step 3.} We complete the proof.
Set $\kappa_0=\frac12\min\{m,1-\delta\}>0$.
For $\xi(t_0)\leq\zeta\leq c_*$, $z\geq0$ and $0<\varepsilon\leq\varepsilon_0$, we have
\begin{equation}\label{eq4.70}
\begin{aligned}
&\delta+(1+\varepsilon)[Q(\zeta,z+\varepsilon/2)-\delta]-Q(\zeta,z)\\
&\quad=\int_z^{z+\varepsilon/2}Q_z(\zeta,s)\,ds
+\varepsilon[Q(\zeta,z+\varepsilon/2)-\delta]
\geq\kappa_0\varepsilon.
\end{aligned}
\end{equation}
Indeed, if $Q(\zeta,z+\varepsilon/2)\leq1-\rho$, then \eqref{eq4.64} gives
\[
\int_z^{z+\varepsilon/2}Q_z(\zeta,s)\,ds\geq\frac{m\varepsilon}{2}
\geq\kappa_0\varepsilon.
\]
If $Q(\zeta,z+\varepsilon/2)>1-\rho$, then
\[
\varepsilon[Q(\zeta,z+\varepsilon/2)-\delta]
>\varepsilon(1-\rho-\delta)>\frac{\varepsilon(1-\delta)}2
\geq\kappa_0\varepsilon.
\]
Fix $\varepsilon\in(0,\varepsilon_0]$. By Lemma~\ref{le4.8}, we may take
$n\geq n_0$ sufficiently large that
\begin{equation}\label{eq4.71}
	\left\{
	\begin{aligned}
		&\bar h_n(0)-h(t_n)=L_0+\varepsilon-l(t_n)\geq\varepsilon/2,\\
		&\sup_{0\leq r\leq h(t_n)}
		|u(t_n,r)-q_*(h(t_n)-r)|+\frac{C_Qc_N}{t_n}
		<\frac{\kappa_0\varepsilon}{2}.
	\end{aligned}
	\right.
\end{equation}
For $0\leq r\leq h(t_n)$, write $z=h(t_n)-r$.
By $Q_z>0$, \eqref{eq4.70} and \eqref{eq4.71},
\[
\begin{aligned}
\bar u_n(0,r)
\geq\delta+(1+\varepsilon)[Q(\xi(t_n),z+\varepsilon/2)-\delta]\geq Q(\xi(t_n),z)+\kappa_0\varepsilon.
\end{aligned}
\]
Hence, by \eqref{eq4.62} and \eqref{eq4.71},
\begin{equation}\label{eq4.72}
	\left\{
	\begin{aligned}
		&h(t_n)<\bar h_n(0),\\
		&\bar u_n(0,r)-u(t_n,r)
		\geq\kappa_0\varepsilon-\frac{C_Qc_N}{t_n}
		-|u(t_n,r)-q_*(h(t_n)-r)|
		>\frac{\kappa_0\varepsilon}{2}.
	\end{aligned}
	\right.
\end{equation}
Suppose that there exists $T>0$ such that
\[
h(t_n+t)<\bar h_n(t),\quad 0\leq t<T,\quad
h(t_n+T)=\bar h_n(T).
\]
Since
\[
(\bar u_n)_r(t,0)
=-(1+\varepsilon e^{-\beta t})Q_z(\xi(t+t_n),\bar h_n(t))
<0=u_r(t_n+t,0),
\]
\eqref{eq4.69}, \eqref{eq4.72} and the maximum principle give, 
\[
u(t_n+t,r)<\bar u_n(t,r) \ \mbox{ for }\ 0\leq t<T  \mbox{ and }  0\leq r\leq h(t_n+t).
\]
At $t=T$, \eqref{eq4.66} gives
\[
\begin{aligned}
\bar h_n'(T)&\leq h'(t_n+T)
=-\mu_0u_r(t_n+T,h(t_n+T))\\
&\leq-\mu_0(\bar u_n)_r(T,\bar h_n(T))<\bar h_n'(T),
\end{aligned}
\]
a contradiction. Hence, for $t\geq0$ and $0\leq r\leq h(t_n+t)$,
\[
h(t_n+t)\leq\bar h_n(t),\quad
u(t_n+t,r)\leq\bar u_n(t,r).
\]
It follows that
\[
l(t_n+t)\leq L_0+\varepsilon+B\varepsilon(1-e^{-\beta t}),\quad t\geq0.
\]
Letting $t\to\infty$ with $n$ fixed, we obtain
\[
\limsup_{s\to\infty}l(s)\leq L_0+(B+1)\varepsilon.
\]
Letting $\varepsilon\downarrow0$ proves \eqref{eq4.60}, and hence $l(t)\to L_0$.
\end{proof}

\begin{proof}[Proof of Theorem \ref{th1.4}]
By Lemma~\ref{le4.9}, $h(t)-c_*t+c_N\log t\to\hat h$ and $l(t)\to L_0$
as $t\to\infty$.
Let $s_n\to\infty$ be arbitrary. Since
$l(s_n)\to L_0=\liminf_{t\to\infty}l(t)$,
Lemma~\ref{le4.8} gives, after passing to a subsequence,
\[
|h'(s_n)-c_*|+\sup_{0\leq r\leq h(s_n)}
|u(s_n,r)-q_*(h(s_n)-r)|\longrightarrow0.
\]
Since $\{s_n\}$ was arbitrary,
\[
|h'(t)-c_*|+\sup_{0\leq r\leq h(t)}
|u(t,r)-q_*(h(t)-r)|\longrightarrow0
\quad\text{as }t\to\infty.\qedhere
\]
\end{proof}
\par\endgroup

\section{Proof of Theorem \ref{th2.1}}
In this section, we complete the proof of Theorem \ref{th2.1}. The proof follows that of \cite{CDNZ}
with suitable modification.

		\begin{proof}
	We finish the proof	via a contraction mapping argument combined with an extension technique. 
		
		\noindent	$\textbf{Step\,1}.$ We straighten the free boundary $r=h(t)$.

		Define
	\[s=\frac{r}{h(t)} \ \mbox{ and }\ U(t, s) := u(t, r).\]
		There exists $T>0$ such that $h(t)>\frac{h_0}{2}$ for $t\in[0, T]$.	
	 Then for $t \in (0, T]$, system \eqref{eq1.1} is equivalent to
		\begin{equation}\label{eq5.1}
			\left\{
			\begin{array}{ll}
				U_{t} - \frac{d}{h^{2}(t)}  [U_{ss}+\frac{N-1}{s}U_s] -\frac{sh^{\prime}(t)}{h(t)}  U_{s} = f(U), & 0 < t \leq T, \, 0 < s < 1, \\
				U_s(t, 0) = 0, \quad U(t, 1) = \delta, & 0 < t \leq T, \\
				h'(t) = -\dfrac{d}{\delta} \dfrac{U_{s}(t, 1)}{h(t)}, & 0 < t \leq T, \\
				U(0, s) = u_{0}\left(h_{0} s\right), & 0 < s < 1,
			\end{array}
			\right.
		\end{equation}
	 System \eqref{eq5.1} is an initial-boundary value problem with fixed boundaries.	
		
		\noindent	$\textbf{Step\,2}.$ An extension technique.

		Define constants
		\[
		H := \max \{1, |h^{0}|\}, \quad 
		K := \max \left\{1, \| u_{0} \|_{C^{2}[0, h_{0}]} \right\}, \quad 
		T_{1} := \frac{h_{0}}{2(H + |h^{0}|)},
		\]
		where $h^{0} = -\frac{d}{\delta} u_{0}'(h_{0})$. For $T \in (0, T_{1}]$ and the domain $D_T := [0, T] \times [0, 1]$, define the function spaces
		\[
		\begin{aligned}
			X_{1,T} &:= \left\{ U \in C(D_T) : U(0, s) = u_0(h_0 s),\ \| U - u_0(h_0 \cdot) \|_{C(D_T)} \leq K \right\}, \\
			X_{2,T} &:= \left\{ h \in C^{0,1}([0,T]) :  h(0) = h_0,\  h'(0) = h^{0}, \ \| h' - h^{0} \|_{L^{\infty}([0,T])} \leq H  \right\}.
		\end{aligned}
		\]
	 The product space $X_T := X_{1,T} \times X_{2,T}$ forms a complete metric space under the metric
		\[
		d\left( (U_1, h_1), (U_2, h_2) \right) := \| U_1 - U_2 \|_{C(D_T)} + \| h'_1 - h'_2 \|_{L^{\infty}([0,T])}.
		\]
		For $0 < T < T_1$, define the subspace $X_{T_1}^T := X_{1,T_1}^T \times X_{2,T_1}^T \subset X_{T_1}$ where
		\[
		\begin{aligned}
			X_{1,T_1}^T &:= \left\{ U \in X_{1,T_1} : U(t,s) = U(T,s) \quad\text{for all}\quad t \in [T, T_1], \, s \in [0,1] \right\}, \\
			X_{2,T_1}^T &:= \left\{ h \in X_{2,T_1} : h(t) = h(T) \quad\text{for all}\quad t \in [T, T_1] \right\}.
		\end{aligned}
		\]
		Any $(U, h) \in X_T$ admits an extension to $X_{T_1}^T$ via the above definition. Consequently, we will identify $X_T$ with $X_{T_1}^T$.
	We extend $f$ to $\mathbb{R}$ by setting $f(u)=f'(0)u$ for $u<0$. Since $f(0)=0$, this extension is of class $C^1$, and we still denote it by $f$.

		For each pair $(U, h) \in X_T = X_{T_1}^T \subseteq X_{T_1}$, consider the initial-boundary value problem 
		\begin{equation}\label{eq5.2}
			\left\{
			\begin{array}{ll}
				\bar{U}_{t} - \frac{d}{h^{2}(t)}[ \bar{U}_{ss}+\frac{N-1}{s}\bar{U}_s ]- \frac{sh^{\prime}(t)}{h(t)}  \bar{U}_{s} = f(U), & 0 < t \leq T_1, \, 0 < s < 1, \\
				\bar{U}_s(t,0) = 0, \, \bar{U}(t,1) = \delta, & 0 < t \leq T_1, \\
				\bar{U}(0,s) = u_0(h_0 s), & 0 < s < 1.
			\end{array}
			\right.
		\end{equation}
		Direct calculations yield 
		\begin{equation}\label{eq5.3}
			4d/(9h_0^2) \leq d / h^{2}(t) \leq 4d/h_0^2.
		\end{equation}
		For $S_1 = (t_1, s_1)$ and $S_2 = (t_2, s_2)$ in $D_{T_1}$ with $\delta(S_1, S_2) = \sqrt{(s_1-s_2)^2 + |t_1 - t_2|}$, we have
		\begin{equation}
			\begin{aligned}
				\omega(R) := d \sup_{\delta(S_1,S_2) \leq R} \left| \frac{1}{h^2(t_1)} - \frac{1}{h^2(t_2)} \right| &\leq d \sup_{\delta(S_1,S_2) \leq R} \frac{|h^2(t_2) - h^2(t_1)|}{h^2(t_1)h^2(t_2)} \\
				&\leq \frac{48d(|h^0| + H)}{h_0^3} R^2 \to 0 \quad \text{as } R \to 0.
			\end{aligned}
		\end{equation}
		Additionally, we have
		\begin{equation}\label{eq5.5}
		 |sh'(t) / h(t)| \leq 2(|h^0| + H)/h_0.
		\end{equation}
		Let $B_1:=\{x\in\mathbb{R}^N:|x|<1\}$ and $Q_{T_1}:=(0,T_1)\times B_1$. We rewrite \eqref{eq5.2} as
		\[
		\begin{cases}
		\bar{U}_t-\dfrac{d}{h^2(t)}\Delta\bar{U}-\dfrac{h'(t)}{h(t)}x\cdot\nabla\bar{U}=f(U), & (t,x)\in Q_{T_1},\\
		\bar{U}=\delta, & (t,x)\in(0,T_1)\times\partial B_1,\\
		\bar{U}(0,x)=u_0(h_0|x|), & x\in B_1.
		\end{cases}
		\]
		By \eqref{eq5.3}--\eqref{eq5.5}, the standard $L^p$ theory (see \cite[proof of Theorem 4.1]{DG}) and Sobolev embedding yield a unique radial solution satisfying
		\begin{equation}\label{eq5.6}
		\|\bar{U}\|_{C^{\frac{1+\alpha}{2},1+\alpha}(D_{T_1})}\leq C_{T_1}\|\bar{U}\|_{W_p^{2,1}(Q_{T_1})}\leq K_1,
		\end{equation}
		where $p>(N+2)/(1-\alpha)$ and $K_1$ depends only on the fixed data, $p$ and $\alpha$, independently of $T$ and $(U,h)$.
		
		Define $\bar{h}(t)$ via 
		\begin{equation*}
			\bar{h}'(t) = -\frac{d}{\delta} \frac{\bar{U}_s(t, 1)}{h(t)}, \quad \bar{h}(0) = h_0.
		\end{equation*}
		Then $\bar{h}' \in C^{\alpha/2}([0, T_1])$ with
		\begin{equation}
			\| \bar{h}' \|_{C^{\alpha/2}([0, T_1])} \leq K_2,
		\end{equation}
		where $K_2$ depends on $K_1$.
		Define the mapping $ \mathcal{F}: X_{T}=X_{T_{1}}^{T} \rightarrow C\left(D_{T_{1}}\right) \times C([0, T_{1}])$  by
		\[\mathcal{F}(U, h)=(\bar{U}, \bar{h}).\]
		Set
		\[\tilde{\mathcal{F}}(U, h):=\left.\mathcal{F}(U, h)\right|_{X_{T}}.\]
		Then we see that  $(U, h)$  is a fixed point of $ \tilde{\mathcal{F}}$  if and only if it solves \eqref{eq5.1}, which is equivalent to \eqref{eq1.1} for  $t \in[0, T]$.
		
		\medskip
		
		\noindent	$\textbf{Step\,3}.$ We show that $ \tilde{\mathcal{F}} $ is a contraction mapping for small enough  $T>0$.
		
		Given any fixed $ 0<T<\min \left\{T_{1},\,{(\frac{K_{1}}{K})}^{\frac{-2}{1+\alpha}},\, (\frac{K_{2}}{H})^{\frac{-2}{\alpha}}\right\}$, we have
		\[\begin{array}{l}
			\left\Arrowvert\bar{U}-u_{0}(h_0\cdot)\right\Arrowvert_{C\left(D_{T}\right)} \leq T^{\frac{1+\alpha}{2}}\Arrowvert\bar{U}\Arrowvert_{C^{0, \frac{1+\alpha}{2}}\left(D_{T}\right)} \leq T^{\frac{1+\alpha}{2}}\Arrowvert\bar{U}\Arrowvert_{C^{0, \frac{1+\alpha}{2}}\left(D_{T_{1}}\right)} \leq K_{1} T^{\frac{1+\alpha}{2}} \leq K, \\
			\left\Arrowvert\bar{h}^{\prime}(t)-h^{0}\right\Arrowvert_{L^{\infty}([0, T])} \leq T^{\frac{\alpha}{2}}\left\Arrowvert\bar{h}^{\prime}\right\Arrowvert_{C^{\frac{\alpha}{2}}([0, T])} \leq T^{\frac{\alpha}{2}}\left\Arrowvert\bar{h}^{\prime}\right\Arrowvert_{C^{\frac{\alpha}{2}}\left(\left[0, T_{1}\right]\right)} \leq K_{2} T^{\frac{\alpha}{2}} \leq H, \\
		\end{array}\]
		which implies that  $\tilde{\mathcal{F}} $ maps $ X_{T} $ to itself.
		
		Next, we prove that $\tilde{\mathcal{F}}$ is a contraction mapping on $X_T$ for all sufficiently small $T > 0$. Let $(U_i, h_i) \in X_T = X_{T_1}^T$ for $i = 1, 2$, and define
		\[
		W := \bar{U}_1 - \bar{U}_2.
		\]
		Then, it is easy to verify that $W$ satisfies
		\begin{equation}\label{eq5.8}
			\left\{
			\begin{array}{ll}
				W_t - d_1(t) [W_{ss}+\frac{N-1}{s} W_s ]- \beta_1(t) s W_s = \Psi, & 0 < t \leq T_1, \, 0 < s < 1, \\
				W_s(t, 0) = W(t, 1) = 0, & 0 < t \leq T_1, \\
				W(0, s) = 0, & 0 \leq s \leq 1,
			\end{array}
			\right.
		\end{equation}
		where
		\begin{align*}
			\Psi :=   \left( \beta_1(t) s -\beta_2(t) s \right) \bar{U}_{2,s} 
			+ \left( d_1(t) - d_2(t) \right) \left(\bar{U}_{2,ss}+\frac{N-1}{s}\bar{U}_{2,s}\right) + f(U_1) - f(U_2),
		\end{align*}
		with $d_i(t) = d /h_i^{2}(t)$, $\beta_i(t) = h_i'(t) / h_i(t)$ for $i = 1, 2$.

		Direct calculation gives
		\begin{equation}\label{eq5.9}
			\begin{aligned}
				\left\Vert \beta_{1}(t)s - \beta_{2}(t)s \right\Vert_{L^\infty(D_{T_1})} 
				= \mathop{\rm ess\,sup}_{t\in[0,T_1]} \left| \frac{h_1'(t)}{h_1(t)} - \frac{h_2'(t)}{h_2(t)} \right| 
				\leq M_1 \left\Vert h_1 - h_2 \right\Vert_{C^{0,1}([0, T])}
			\end{aligned}
		\end{equation}
		for some constant $M_1$ depending on $h_0$, $h^0$ and $T_1$, but independent of $T$.  Similarly,
		\begin{equation}\label{eq5.10}
			\begin{aligned}
				\left\Vert d_1(t) - d_2(t) \right\Vert_{C(D_{T_1})} 
				= \sup_{(t, y) \in D_{T_1}} \left| \frac{d}{h_1^2(t)} - \frac{d}{h_2^2(t)} \right| \leq M_2 \left\Vert h_1 - h_2 \right\Vert_{C^{0,1}([0, T])}
			\end{aligned}
		\end{equation}
		where the positive constant $M_2$ depends on $h_0$ , $d$ and $T_1$, but is independent of $T$.
		Combining \eqref{eq5.6}, \eqref{eq5.9} and \eqref{eq5.10},  we obtain
		\begin{equation}
			\begin{aligned}
				\left\Vert \Psi \right\Vert_{L^p(Q_{T_1})} 
				&\leq \left\Vert \beta_1 s - \beta_2 s \right\Vert_{L^\infty(D_{T_1})} \left\Vert \bar{U}_{2,s} \right\Vert_{L^p(Q_{T_1})} \\
				&\quad + \left\Vert d_1 - d_2 \right\Vert_{L^\infty(D_{T_1})} \left\Vert \bar{U}_{2,ss}+\frac{N-1}{s}\bar{U}_{2,s} \right\Vert_{L^p(Q_{T_1})}  + \left\Vert f(U_1) - f(U_2) \right\Vert_{L^p(Q_{T_1})} \\
				&\leq M_3 \left( \left\Vert U_1 - U_2 \right\Vert_{C(D_{T_1})} + \left\Vert h_1 - h_2 \right\Vert_{C^{0,1}([0, T_1])} \right),
			\end{aligned}
		\end{equation}
		where $M_3$ depends only on $Q_{T_1}$, $p$, $K_1$, $f$, $M_1$, and $M_2$.
		
		Applying $L^p$ estimates to  \eqref{eq5.8} on $Q_{T_1}$ and the Sobolev embedding theorem we obtain
		\begin{equation}\label{eq5.12}
			\begin{aligned}
				\left\Vert W \right\Vert_{C^{\frac{1+\alpha}{2},1+\alpha}(D_{T_1})} 
				\leq C_{T_1} \left\Vert W \right\Vert_{W_p^{2,1}(Q_{T_1})} \leq K_3 \left( \left\Vert U_1 - U_2 \right\Vert_{C(D_{T_1})} + \left\Vert h_1 - h_2 \right\Vert_{C^{0,1}([0,T_1])} \right),
			\end{aligned}
		\end{equation}
		with $K_3$ depending on $p$, $Q_{T_1}$, $M_3$ and $C_{T_1}$. By the definition of $\bar{h}_i'$, we have
		\[
		\bar{h}_1'(t)-\bar{h}_2'(t)=-\frac{d}{\delta}\left[\frac{W_s(t,1)}{h_1(t)}+\bar{U}_{2,s}(t,1)\left(\frac{1}{h_1(t)}-\frac{1}{h_2(t)}\right)\right].
		\]
		Since $h_i\geq h_0/2$ and $h_i$ have a uniform Lipschitz bound on $[0,T_1]$, then 
		\[
		\left\Vert h_i^{-1}\right\Vert_{C^{\alpha/2}([0,T_1])}\leq C,\qquad
		\left\Vert h_1^{-1}-h_2^{-1}\right\Vert_{C^{\alpha/2}([0,T_1])}\leq C\left\Vert h_1-h_2\right\Vert_{C^{0,1}([0,T_1])}.
		\]
		Therefore, \eqref{eq5.6} yields
		\begin{equation}\label{eq5.13}
			\left\Vert \bar{h}_1' - \bar{h}_2' \right\Vert_{C^{\alpha/2}([0,T_1])} \leq C \left(\left\Vert W \right\Vert_{C^{\frac{1+\alpha}{2},1+\alpha}(D_{T_1})}+\left\Vert h_1-h_2\right\Vert_{C^{0,1}([0,T_1])}\right),
		\end{equation}
		where $C$ depends on $d$, $\delta$, $h_0$, $h^0$, $T_1$, $\alpha$ and $K_1$, but is independent of $T$. Since $h_1(0)=h_2(0)$, we have
		\[
		\left\Vert h_1-h_2\right\Vert_{C^{0,1}([0,T_1])}\leq (1+T_1)\left\Vert h_1'-h_2'\right\Vert_{L^\infty([0,T_1])}.
		\]
		The inequalities \eqref{eq5.12} and \eqref{eq5.13} imply
		\begin{align*}
			\left\Vert \bar{U}_1 - \bar{U}_2 \right\Vert_{C^{\frac{1+\alpha}{2},1+\alpha}(D_{T_1})} 
			+ \left\Vert \bar{h}_1' - \bar{h}_2' \right\Vert_{C^{\alpha/2}([0,T_1])} 
			\leq K_4 \left( \left\Vert U_1 - U_2 \right\Vert_{C(D_{T_1})} 
			+ \left\Vert h_1' - h_2' \right\Vert_{L^\infty([0,T_1])} \right),
		\end{align*}
		where $K_4$ depends on $C$, $K_3$ and $T_1$, but is independent of $T$.
		If we choose
		\[
		T = \min \left\{ \frac{1}{2},\, T_1,\, (\frac{K_1}{K})^{-\frac{2}{1+\alpha}},\, (\frac{K_2}{H})^{-\frac{2}{\alpha}},\, (2K_4)^{-\frac{2}{\alpha}} \right\},
		\]
		then, using $W(0,s)=0$ and $\bar{h}_1'(0)=\bar{h}_2'(0)=h^0$, we obtain
		\begin{align*}
			\left\Vert \bar{U}_1 - \bar{U}_2 \right\Vert_{C(D_T)} 
			+ \left\Vert \bar{h}_1' - \bar{h}_2' \right\Vert_{C([0,T])} 
			&  \leq T^{\frac{1+\alpha}{2}} \left\Vert \bar{U}_1 - \bar{U}_2 \right\Vert_{C^{\frac{1+\alpha}{2},1+\alpha}(D_{T_1})} 
			+ T^{\frac{\alpha}{2}} \left\Vert \bar{h}_1' - \bar{h}_2' \right\Vert_{C^{\alpha/2}([0,T_1])} \\
			&\leq K_4 T^{\frac{\alpha}{2}}\left(\left\Vert U_1-U_2\right\Vert_{C(D_T)}+\left\Vert h_1'-h_2'\right\Vert_{L^\infty([0,T])}\right) \\
			&  \leq \frac{1}{2} \left( \left\Vert U_1 - U_2 \right\Vert_{C(D_T)} 
			+ \left\Vert h_1' - h_2' \right\Vert_{L^\infty([0,T])} \right).
		\end{align*}
		Thus $\tilde{\mathcal{F}}$  is a contraction mapping on $ X_{T}$. Therefore, it has a unique fixed point  $(U, h) $ in $ X_{T}$. The maximum principle implies that $U > 0$ in $[0,T] \times (0,1]$.
		
		 Finally, we apply the Schauder theory to obtain additional regularity for the solution of \eqref{eq5.1} in $ \left(0, T\right] \times[0, 1] $ as in \cite{CDNZ} to obtain that for any $\varepsilon>0$ small there exists $M_\varepsilon>0$ such that
		 	\[
		 	\|U\|_{C^{1+\frac{\alpha}{2}, 2+\alpha}([\varepsilon, T] \times [0, 1])} \leq M_\varepsilon.
		 \]
		We deduce that $h \in C^{1+\frac{1+\alpha}{2}}((0, T])$ and $u \in C^{1+\frac{\alpha}{2}, 2+\alpha}(\Omega_T)$. Therefore, $(u, h)$ constitutes a classical solution of \eqref{eq1.1} on $\Omega_T$.

	\end{proof}

\end{document}